\documentclass{amsart}

 \usepackage{amsmath,amssymb,xcolor,stmaryrd,nicefrac,mathtools,multicol,leftindex}
\usepackage[inline]{enumitem}

\newtheorem{MainTheo}{Theorem}
\newtheorem{theorem}{Theorem}[section]
\newtheorem{lemma}[theorem]{Lemma}

\theoremstyle{definition}
\newtheorem{definition}[theorem]{Definition}
\newtheorem{example}[theorem]{Example}

\theoremstyle{remark}
\newtheorem{remark}[theorem]{Remark}

\numberwithin{equation}{section}

\newtheorem{innercustomthm}{Theorem}
\newenvironment{customthm}[1]
  {\renewcommand\theinnercustomthm{#1}\innercustomthm}
  {\endinnercustomthm}

\newcommand{\bfun}{o^*}
\newcommand{\se}[2]{\leftindex[I]^{#2}#1}
\newcommand{\odh}{\mathcal D}
\newcommand{\ignore}[1]{{\color{gray}[Skipped text]}}
\newcommand{\pref}[1]{(\ref{#1})}
\newcommand{\ns}{\oplus}
\newcommand{\col}{:}
\newcommand{\ofun}{O}

\newcommand{\mco}[1]{#1^*}
\newcommand{\us}{_}

\newcommand{\baseB}[1]{{\rm base}_{ B}(#1)}
\newcommand{\basep}[2]{{\rm base}_{#1}(#2)}
\newcommand{\Base}{{\rm Base}}

\newcommand{\coeffs}[2]{{\rm Digit}_{#1}(#2)}
\newcommand{\chgbases}[2]{ 
\left ( \begin{smallmatrix}
  #2\\
  #1
\end{smallmatrix}
\right )
}

\newcommand{\Chgbases}[2]{ 
\langle \begin{smallmatrix}
  #2\\
  #1
\end{smallmatrix}
\rangle
}

\newcommand{\upgrade}[1]{{\uparrow}_{#1}}
\newcommand{\ug}{{\uparrow}}

\newtheorem{prop}[theorem]{Proposition}
\newtheorem{claim}[theorem]{Claim}

\newtheorem{proposition}[theorem]{Proposition}
\newtheorem{convention}[theorem]{Convention}

\DeclareFontFamily{U}{mathx}{\hyphenchar\font45}
\DeclareFontShape{U}{mathx}{m}{n}{
      <5> <6> <7> <8> <9> <10>
      <10.95> <12> <14.4> <17.28> <20.74> <24.88>
      mathx10
      }{}
\DeclareSymbolFont{mathx}{U}{mathx}{m}{n}
\DeclareFontSubstitution{U}{mathx}{m}{n}
\DeclareMathAccent{\widecheck}{0}{mathx}{"71}
\DeclareMathAccent{\wideparen}{0}{mathx}{"75}

\newcommand{\fs}[2]{ #1 [ #2]}

\newcommand{\fsi}[2]{ {#1}\llbracket{#2}\rrbracket }

\newcommand{\goodp}[3]{{\G}^{#3}_{#2}(#1)}
\newcommand{\G}{\mathbb G}

\newlength{\eparindent}

\newlist{CasesA}{enumerate}{1}
\setlist[CasesA,1]{label={\sc Case \Alph*},wide, labelwidth=!, labelindent=0pt,listparindent=\eparindent}

\newlist{Cases}{enumerate}{9}
\setlist[Cases,1]{label={\sc Case {\rm \arabic*}},wide, labelwidth=!, labelindent=0pt,listparindent=\eparindent}
\setlist[Cases,2]{label*= {\rm .\arabic*},wide, labelwidth=!, labelindent=0pt,listparindent=\eparindent}
\setlist[Cases,3]{label*= {\rm .\arabic*},wide, labelwidth=!, labelindent=0pt,listparindent=\eparindent}
\setlist[Cases,4]{label*= {\rm .\arabic*},wide, labelwidth=!, labelindent=0pt,listparindent=\eparindent}
\setlist[Cases,5]{label*= {\rm .\arabic*},wide, labelwidth=!, labelindent=0pt,listparindent=\eparindent}
\setlist[Cases,6]{label*= {\rm .\arabic*},wide, labelwidth=!, labelindent=0pt,listparindent=\eparindent}
\setlist[Cases,7]{label*= {\rm .\arabic*},wide, labelwidth=!, labelindent=0pt,listparindent=\eparindent}
\setlist[Cases,8]{label*= {\rm .\arabic*},wide, labelwidth=!, labelindent=0pt,listparindent=\eparindent}
\setlist[Cases,9]{label*= {\rm .\arabic*},wide, labelwidth=!, labelindent=0pt,listparindent=\eparindent}

\newcommand{\begincases}{\begin{enumerate}[label*={\sc Case \arabic*},wide, labelwidth=!, labelindent=0pt]}

\newcommand{\begincasesast}{\begin{enumerate*}[label*={\sc Case \arabic*},labelwidth=!, labelindent=0pt]}
\newcommand{\beginclaims}{\begin{enumerate}[label*={\sc Claim },wide, labelwidth=!, labelindent=0pt]}
\newcommand{\begincasesa}{\begin{enumerate}[label={\sc Case ({\rm \roman*})},wide, labelwidth=!, labelindent=0pt]}
\newcommand{\beginsubcases}{\begin{enumerate}[label*= {\rm .\arabic*},wide, labelwidth=!, labelindent=0pt]}
\newcommand{\beginsubcasesast}{\begin{enumerate*}[label*= {\rm .\arabic*},wide, labelwidth=!, labelindent=0pt]}

\newcommand{\putaway}[1]{}

\newcommand{\ve}{\varepsilon}

\renewcommand{\phi}{{\overline{\varphi}}}

\newcommand{\Om}{{\Omega}}

\newcommand{\al}{{\alpha}}

\newcommand{\be}{{\beta}}

\newcommand{\ga}{\gamma}

\newcommand{\de}{\delta}

\newcommand{\la}{\lambda}

\newcommand{\om}{\omega}

\begin{document}

\title{The Fractal Goodstein Principle}

\author{David Fernández-Duque}
\address{Department of Philosophy, University of Barcelona, Spain}
\email{fernandez-duque@ub.edu}
\thanks{The first author was supported by the Spanish Ministry of Science and Innovation grant PID2023-149556NB-I00.}

\author{Andreas Weiermann}
\address{Department of Mathematics WE16, University of Ghent, Belgium}
\email{andreas.weiermann@ugent.be}





\begin{abstract}
 The original Goodstein process is based on writing numbers in hereditary $b$-exponential normal form: each number $n$ is written in some base $b\geq 2$ as $n=b^ea+r$, with $e$ and $r$ iteratively being written in hereditary $b$-exponential normal form.
We define a new process which generalises the original by writing expressions in terms of a hierarchy of bases $B$, instead of a single base $b$.
In particular, the `digit' $a$ may itself be written with respect to a smaller base $b'$.
We show that this new process always terminates, but termination is independent of Kripke-Platek set theory, or other theories of Bachmann-Howard strength.
\end{abstract}

\maketitle

\section{Introduction}

Goodstein's principle~\cite{Goodsteinb} is a classic example of a `purely number-theoretic' result independent of Peano arithmetic ($\sf PA$)~\cite{Kirby}.
While the statement is relatively simple, the proof uses the well-foundedness of $\varepsilon_0$ \cite{Goodstein1944}, the proof-theoretic ordinal of $\sf PA$.
The principle is based on {\em hereditary exponential normal forms,} where\-by each natural number is written in some base $b\geq 2$ in the standard way, but then each exponent appearing in the expansion is also written in base $b$.
For example, we would write $6$ in base $2$ as $6=_2 2^2+2$.
The Goodstein process then proceeds by increasing the base by one, then subtracting one. If the first element in the sequence is $6$, we first increase its base by one, obtaining $30=_33^3+3$, so the second element in the sequence is $29$.
Despite initially growing rather quickly, this process eventually reaches zero, a fact not provable in $\sf PA$.

Several recent results have extended Goodstein's theorem in various ways \cite{FSGoodstein,FernandezWCiE,FernandezWWalk}, including variants independent of Kripke-Platek set theory ($\sf KP$) or other theories of Bachmann-Howard strength 
\cite{AraiWW,FernandezFastWalks}.
The latter are themselves based on functions not provably total in $\sf PA$, with the general philosophy being that they are fair game, as long as they are provably total in $\sf KP$.
But this begs the question: are such functions really necessary for Goodstein processes of high proof-theoretic strength?
May it be possible to squeeze some more power out of Goodstein's original construction?
A partial answer is `no', because the authors~\cite{FernandezWWalk} have shown that no way of writing numbers using elementary functions will lead to termination times larger than those given by Goodstein's original progression.
However, there is a way to circumvent this restriction.

Our first observation is that there is no need to always increase the base by one; the next base may be chosen to be much larger than the previous one.
However, simply replacing the base $b$ by some uniform $f(b)$ will not yield a principle of higher proof-theoretic strength, since the well-foundedness of $\varepsilon_0$ may still be applied.
This leads to our second observation: the new base for each $n \in \mathbb N$ could depend on $n$.
We may consider, say, replacing the base of $n$ by $n$ itself.
Such an approach will not lead to a uniform base being used at the next step, since following this rule would lead us to replacing the base of $n+1$ by $ n+1$, and thus every natural number would be regarded as a possible base at each stage.
This na\"ive idea would not quite work out, but we can follow an intermediate approach and create new bases only at `critical' numbers.
This will lead to a hierarchy of bases at each stage, rather than a single `global' base.
One can even imagine that for a number written as $b^ea+r$ in base $b$, the numbers $b,e,a,r$ may all be written in distinct bases, with {\em their} digits written in terms of even smaller bases, each performing its own Goodstein process, like a self-similar fractal made up of smaller copies of itself.
As we will see, this can be implemented in a   manner that yields a terminating sequence; in this way, we obtain a Goodstein principle very close to the classic version, but of much higher proof-theoretic strength.

Let us give a very brief description of $\sf KP$; see e.g.~\cite{Barwise} for more details.
The theory $\sf KP$ (with infinity) is axiomatised by all axioms of $\sf ZFC$ except for powerset, but with separation restricted to $\Delta_0$ formulas and replacement restricted to $\Delta_0$-collection.
Roughly, new sets may only be defined in terms of the elements of previously defined sets.
Other theories of the same proof-theoretic strength as $\sf KP$ are the theory ${\sf ID}_1$ of non-iterated inductive definitions and ${\Pi}^1_1$-${\sf CA}^-_0$ of parameter-free ${\Pi}^1_1$ comprehension.
They share the proof-theoretic ordinal $\vartheta(\ve_{\Om+1})$, where $\Om$ is the first uncountable cardinal and $\vartheta \colon \ve_{\Om+1} \to \Om$ `collapses' possibly uncountable ordinals into countable ones~\cite{RathjenFragments}.
This theory is far more powerful than $\sf PA$ and certainly proves Goodstein's theorem, but not its fractal extension.

\section{Fractal Goodstein Processes}

Let us recall some notions from the classic Goodstein process.
Letting $b\in\mathbb N\setminus \{0,1\}$ and $0<n\in \mathbb N$, we define the {\em $b$-decomposition} of $n$ to be $b^ea+r $, where $a,e,r$ are uniquely chosen so that $0<a<b$ and $r<b^e$, in which case we write $n=_b b^ea+r $.
Technically, we think of the $b$-decomposition of $n$ as the quadruple $(b,e,a,r)\in\mathbb N^4$, but one could iteratively write  $e$ and $r$ using their own $b$-decompositions, leading to Goodstein's hereditary exponential normal forms.
Note that $=_b$ is not symmetric, i.e.~$\tilde b^{\tilde e} \tilde a +\tilde r =_B b^ea+r$ does not imply $b^ea+r  =_b  \tilde b^{\tilde e} \tilde a +\tilde r$; only the right-hand side needs to be a $b$-decomposition.
The number $a$ is regarded as a `digit,' since it is less than the base.
Formally, for $b > 0$, we say that $c$ is a {\em $b$-digit} of $n$ if   $n=_b b^ea+r$, and either $c=a$, or $c$ is a $b$-digit of $e$ or $r$. Zero is defined to have no $b$-digits, and   if $0<n<b$,  $n=_b b^0n+0$, so $n$ is the only $b$-digit of $n$.
The set of $b $-digits of $n$ is denoted $\coeffs bn$.

The classic Goodstein process is based on the base-change operator, where we replace $b$ in the $ b$-decomposition of $n$ by some other base $c \geq b$.
This is defined recursively as $\chgbases b{c}0 = 0$ and $\chgbases b{c}n = c^{\chgbases b{c}e}a+ \chgbases b{c}r$ if $n >0$ and $n=_b b^ea+r$.
The base change has the following basic properties, which may be verified directly, but see e.g.~\cite{FernandezWWalk}. 

\begin{lemma}\label{lemmBCProps}
Let $n,m,b,c,d\in\mathbb N$ with $2\leq b\leq c$.

\begin{enumerate}

\item If $n<m$, then $\chgbases b{c}n <\chgbases b{c}m$.

\item If $c\leq d$, then $\chgbases b{c}n \leq \chgbases b{d}n$.
The inequality is strict if $n\geq b$ and $d>c$.

\item\label{itBCPropsThree} If $n=_b b^ea+r$, then $\chgbases b{c}n =_c c^{\chgbases b{c}e}a+ \chgbases b{c}r$.
\end{enumerate}

 \end{lemma}

Then, the classic Goodstein sequence starting on $n$ is the sequence $(n_i)_{i< I}$, where $n_0=n$, $n_{i+1} = \chgbases{i+2}{i+3}n_i-1$ if $n_i>0$, and either $I$ is finite and $n_{I-1} =0$, or else $I=\infty$ and all elements of the sequence are positive.
Goodstein's theorem then states that $I$ will always be finite~\cite{Goodstein1944}.

We generalize this to a multi-base setting, using {\em base hierarchies.}
Let $B\subseteq \mathbb N \setminus\{0,1\}$.
For $n\in\mathbb N \cup \{\infty\}$, let $S_B(n) $ be the least $b\in B$ such that $b>n $ and let $S_B(n) = \infty $ if there is no such $b$.
We regard every positive integer to divide $\infty$.
 
\begin{definition}
A set $B\subseteq \mathbb N\setminus \{0,1\}$ is a {\em base hierarchy} if $B\neq\varnothing$ and for every $b\in B$, $b\mid S_B(b)$.

If $B$ is a base hierarchy, elements of $B$ are {\em bases.}
\end{definition}

Each $n\in\mathbb N$ is assigned a base from $B$.
The general rule is that the base of $n$ should be the largest base below $n$.
For example, if $B = \{5,10\}$, the number $7$ would be written in base $5$, and (say) $14$ in base $10$.
However, it is unclear if we wish to write $10$ in base $5$ or in base $10$.
Either choice is reasonable, and both will be useful, but writing $10$ in base $5$ will carry some essential information.
We thus assign two bases to $10$: $\baseB {10} = 5$ and $\Base_B(10) = 10$.
On the other hand, $5$, being the least base, will only be assigned the base $5$.

\begin{definition}\label{defBase}
Given a base hierarchy $B$ and $n\in \mathbb N$, we define
\begin{enumerate}

\item the {\em lower base} of $n$ by
\[\baseB n = \max \{b\in B: b \leq \max \{n-1,\min B\}\};\]

\item the {\em upper base} of $n$ by
\[\Base_B (n) = \max \{b\in B: b \leq \max \{n,\min B\}\}.\]

\end{enumerate}
\end{definition}

It is not hard to check that $\baseB n = \Base_B(n)$, except when $n \in B \setminus \{\min B\}$.
For the most part, $\Base_B(n)$ can be regarded as the `official' base of $n$, but $\baseB n$ will be useful in Section~\ref{secTerm} for proving that our Goodstein process is terminating.

\begin{example}\label{exSing}
Every singleton $\{b\}$ with $b\geq 2$ is a base hierarchy, since in this case we have decreed $b\mid \infty =S_B(b)$.
\end{example}

Our Goodstein principle is based on the {\em upgrade operator} $\ug = \ug_B^C$, which takes the role of $\chgbases bc  $ in Goodstein's original theorem.
The purpose of this operator is to transform a number $n$ represented with respect to a base hierarchy $B$ to a new number $\ug n$ represented with respect to another base hierarchy $C$.
Intuitively, given $n\in\mathbb N$ with $b = \Base_B(n)$, in order to perform the upgrade, we first recursively upgrade all the $b$-digits of $n$.
We then search for the least base $c$ for which changing the base $b$ to $c$ will yield a `reasonable' result: namely, we need $\ug n$ to be larger than $\ug(n-1)$ to make $\ug$ monotone, but we also want the result to have $c$ as its base, so the end result should be smaller than $S_C(c)$.
This process of upgrading the digits first and then changing the base $b$ of $n$ to some $c\in C$ is the {\em deep base change,} and the result is denoted by  $\Chgbases  bc n$.
Let us make these operations precise.

\begin{definition}\label{defChg}
Let $B,C$ be base hierarchies with $ \min B\leq \min C $ and $n\in\mathbb N$.
We define $\ug n = \upgrade B^C n\in \mathbb N \cup\{\infty\} $ recursively on $n$.
If $n<\min B$, then $\ug  n = n$.

Otherwise, $n\geq \min B$.
Let $b =\Base_B(n)$ and assume inductively that $\ug m$ is defined for all $m<n$.
We first define the {\em deep base-change operator} $\Chgbases bc = \Chgbases bc^C_B $ on $m\in\mathbb N$:
if $m < b$, then $\Chgbases bc m = \ug m $.
If $m\geq b$, write $m=_b b^ea+r$, and set
\[\Chgbases bc m = c^{\Chgbases bc e } \ug a +\Chgbases bc r .\]
This has the effect of applying the upgrade operator to all the $b$-digits of $m$, then changing the base to $c$.

Then, define  $ \ug n = \Chgbases  bc n $, where $c $ is the least element of $C$ such that
\[\ug(n-1) < \Chgbases  bc n < S_C (c).\]
If no such $c$ exists, $\ug n = \infty$.
\end{definition}

When $n$ is a base in $B$, note that $b=n$ and $\Chgbases bcn=c$.
So, in this case, we basically jump to the next available base in $C$; for this to be a finite value, $C$ should have `enough available bases.'
In order for the upgrade to be well-behaved, we also do not want this value to be `much larger' than $\ug (n-1)$.
The precise conditions required from $C$ are given by Definition~\ref{defGoodSucc} below.

\begin{convention}
If in a result or proof we state that transformations are from $B$ to $C$, this indicates that $\ug$ is to be understood as $\ug^C_B$ and $\Chgbases\cdot\cdot$ as $\Chgbases\cdot\cdot^C_B$.
Similarly, we may state e.g.~that $\ug$ is from $B$ to $C$ to indicate that $\ug=\ug^C_B$.
More often than not, these indices can be suppressed, but there will be occasions when we do consider more than two base hierarchies in the same context.
\end{convention}

\begin{convention}\label{convOpers}
The transformations $\ug $, $\chgbases bc$ and $\Chgbases bc$ are performed before addition but after multiplication and exponentiation.
Thus for example, $\chgbases bc b^ea+r$ should be understood as $(\chgbases bc (b^ea))+r $.
\end{convention}

\begin{example}\label{exUg}
Consider transformations between $B =\{ 2,6  \}$ and $C=\{ 3,30, 180 \} $ and let us compute some values of the upgrade operator.
We have that $\ug 0=0$ and $\ug 1=1$.
Then, $\Base_B(2) = 2$ and $\Chgbases 2 3 2 = 3 <30=S_C(3)$, so $\ug 2= 3$.
More generally, when $B$ and $C$ are base hierarchies with $\min C \geq \min B$, $\ug ^C_B \min B = \min C$.

Now, let $n = 6^{6}+6^{2^2+1}$ and let us compute $\upgrade  n$.
We first perform the upgrade on $2^2+1$.
In this case, $\ug (2^2+1) = 3^3+1 <30 $, exactly as we would expect from the original Goodstein process.

However, when we attempt to replace $6$ by $30$ in $n$ and upgrade the digit, we obtain $30^{30}+30^{3^3+1}$, which is larger than the next base of $C$.
Thus, we instead use the base $180$, and
\[\ug n = 180^{180}+180^{3^3+1}.\]
\end{example}

\begin{definition}\label{defGoodSucc}
If $B,C\subseteq \mathbb N$ and $\ug=\ug_B^C$, we say that $C$ is a {\em good successor} for $B$ if both $B$ and $C$ are base hierarchies and
\begin{enumerate}

\item $\min B \leq \min C$,

\item $\ug n<\infty$ for all $n\in\mathbb N$, and

\item\label{itGoodSuccThree} if $\min B< n+1 \in B$, there are no multiples of $ \Base_ C(\ug n)$ in the interval 
$ ( \ug n , S_C(\ug n))$.

\end{enumerate}
\end{definition}

In other words, a good successor is one for which the upgrade operator is well-defined and bases always jump to the least available multiple of the previous base.
Our main results will involve Goodstein processes defined on sequences of good successors, so it is a good idea to check that these actually exist.
Below, we will define a rather general construction to build good successors of arbitrary bases.
The intuition is that good successors are typically constructed iteratively, in tandem with $\ug$.
That is, if $B$ is a base hierarchy, we build a good successor as an increasing union of finite sets $C = \bigcup_{n\in\mathbb N} C_n$, in such a way that $\ug_B^C x = \ug_B^{C_n} x$ whenever $x\leq n$.
If $n+1\in B$, we add a new base to $C_{n+1}$ to ensure that the good successor condition holds up to $n+1$.
We may also add new bases in cases where $n+1\notin B$ if we desire, but in order for the upgrade to behave nicely, it is convenient to only do so when $n+1$ is a multiple of its own base.
Such a number is {\em $B$-critical,} and the set of $B$-critical numbers is denoted $\partial B$.
The picture is then that we choose some $X$ with $B\subseteq X\subseteq\partial B$, and add a new base to $C_n$ only when $n\in X$.
Let us make this precise.

\begin{definition}\label{defCanon}
Let $B$ be a base hierarchy.
Say that $n>0$ is {\em $B$-critical} if $\Base_B (n) \mid n$.
Let $\partial B$ denote the set of $B$-critical numbers.

Given $X$ with $B\subseteq X\subseteq \partial B$, we define a new base hierarchy $B^X = \bigcup_{n <\infty} B^X_n  $.
Let $n\in\mathbb N$ and assume inductively that $B^X_{m} $ has been defined and is finite for all $m<n$.
Then, we define $B^X_n$ according to the following cases.\smallskip
\begin{Cases}

\item {\rm ($n\leq \min B$).} $B^X_n  = \{ \min B  + 1\}$.
\smallskip

\item {\rm ($\min B< n\notin  X $).} $B ^X_{n } =   B^X_{n-1} $.\smallskip

\item {\rm ($\min B<   n\in X $).}
Let $c=\max  B^X_{n-1}  $ and transformations be from $B$ to $B^X_{n-1}$.
Let $a$ be least so that $ \ug (n-1) < ca$. Then, set\footnote{As a failsafe, we may set $B ^X_{n } =   B^X_{n-1} $ if $\ug (n-1) = \infty$; however, as we will see, this will never be the case.}
$B^X_{n } =   B^X_{n-1} \cup \big \{ ca\} $.

\end{Cases}

\medskip

If $X=B$, we define $B'_n=B^X_n$ and $B'=B^X$.
This is the {\em minimalistic successor} of $B$.
Likewise, if $X=\partial B$, we define $B^+_n=B^X_n$ and $B^+=B^X$ and call this the {\em greedy successor} of $B$.
\end{definition}

The new base hierarchy $B^X$ is always a good successor for $B$.
We defer the proof until Proposition~\ref{propMonC}, as this requires some machinery.
Note that the map $X\mapsto  B^X $ is not monotone, so $B'$ is not minimal under set inclusion.
It is, however, the `leftmost' good successor of the form $B^X$, if we identify bases with binary reals $(r_i)_{i\in\mathbb N}$, with $r_i=1$ if $i\in B$, otherwise, $r_i=0$.
Similarly, $B^+$ is the rightmost successor that is of the form $B^X$.

\begin{remark} 
Good successors   of the form $B^X$ are by no means exhaustive: for example, we can let $\min C>\min B+1$ but otherwise follow Definition~\ref{defCanon} to obtain a good successor to the left of $B'$.
Likewise, we can modify the case for $n\in X$ to add any finite number of bases $ca,2ca,4ca,\ldots$ and obtain a good successor which is to the right of $B^+$.
However, good successors of the form $B^X$ (or even just $B^+$) will suffice for establishing our main results, so we will not explore other constructions further in this text.
\end{remark}

\begin{example}\label{exSucc}
Let $B=\{2,6\}$, and let us compute $  B'$.
First, we see that $  B'_n = \{\min B+1\} = \{3\}$ for $n\leq 5$.
From here, we see that $\ug_B^{ B'_5} 5 = 3^3+1 = 28$ and $\Base_ { B'_5} (28) = 3$, so $ B'_6$ is obtained by adding the next multiple of $3$, namely, $30$.
We thus arrive at $ B' = \{ 3,30\}$, and it is not hard to check that $ B'$ is indeed a good successor for $B$.
\end{example}

\begin{example}\label{exPlusOne}
Let $B=\{2\}$ and let us compute $B^+_6$.
Similarly to Example~\ref{exSucc}, $B^+_3  = \{3\}$, but now, to find  $B^+_4$, we first compute $\ug_B^{ B^+_3} 3 = \ug_B^{ B^+_3} (2+1) = 4$.
Since $\Base_{B^+_3}(4) = 3 $, we must add the next multiple of $3$ after $4$, that is, $6$.
We thus arrive at $B^+_4 = B^+_3 \cup \{6\} = \{3,6\}$.
Since $2\nmid 5\notin B$,  $B^+_5  = B^+_4  $.
On the other hand, we do add a new element to $B^+_6$.
Compute $\ug_B^{ B^+_5} 5 = \ug_B^{ B^+_5}(2^2+1) = 6^6+1 $.
The next multiple of $6$ is $6^6+6$, so $B^+_6  = B^+_5  \cup \{6^6+6 \} = \{3,6,6^6+6\} $.

In general, we will add one new base for every even number, and thus $B^+$ is infinite, even though $B$ has only one element.
\end{example}

For our Goodstein processes, each $i\in\mathbb N$ will be endowed with a base hierarchy $\mathcal B_i$.
The collection of these hierarchies is a {\em dynamical hierarchy.}

\begin{definition}
A {\em dynamical hierarchy} is a sequence $\mathcal B=(\mathcal B_i)_{i\in\mathbb N}$, where for each $i\in\mathbb N$, $\mathcal B_i$ is a base hierarchy and $\mathcal B_{i+1}$ is a good successor for $\mathcal B_i$.
\end{definition}

\begin{example}\label{exClassic}
If $ B_i=\{i+2\}$, then $\mathcal B = ( B_i)_{i\in\mathbb N}$ is a dynamical hierarchy.
As we will see, this dynamical hierarchy is related to Goodstein's original theorem.
\end{example}

\begin{example}
Let $ B_0$ be any base hierarchy and define inductively $B_{i+1} = (B')_i $, as in Example~\ref{exSucc}.
Then, $(B_i)_{i\in\mathbb N} $ is a dynamical hierarchy.
If $B_0=\{2\}$, we obtain the dynamical hierarchy of Example~\ref{exClassic} as a special case.
Note moreover that if $B_0$ is finite, then $B_i$ is finite for all $i$.
\end{example}

Our independence results will be based on a specific dynamical hierarchy built using greedy successors.

\begin{definition}\label{defODH}
Let $\odh=(\odh_i)_{i\in\mathbb N}$ be given by $\odh_0 =\{2\} $ and $\odh_{i+1}= \odh^+_{i}$; this is the {\em greedy dynamical hierarchy.}
\end{definition}

If a dynamical hierarchy $\mathcal B$ is given and clear from context, we may replace $\mathcal B_i$ by $i$ in our terminology and notation, e.g.~we may write $=_i$ instead of $=_{\mathcal B_i}$,  $\upgrade i$ instead of $\upgrade{\mathcal B_i}^{\mathcal B_{i+1}}$, etc.
With this, we are ready to define our fractal Goodstein processes.

\begin{definition}
Given a dynamical hierarchy $\mathcal B$ and $n\in\mathbb N$, we recursively define $\goodp ni{\mathcal B} $ by letting $\goodp n0{\mathcal B} =n $ and, if $\goodp ni{\mathcal B}  $ is defined and positive, then $ \goodp n{i+1} {\mathcal B} = \upgrade i \goodp n {i} {\mathcal B} -1 $; otherwise, $ \goodp n{i+1} {\mathcal B}   $ is undefined.
\end{definition}
 
\begin{example}\label{exClassicG}
If $\mathcal B $ is the dynamical hierarchy of Example~\ref{exClassic}, then the sequence $\goodp ni{\mathcal B} $ is the classical Goodstein sequence for $n$.
\end{example}

With this, we may state our main results, proven throughout the rest of the text.

\begin{MainTheo}\label{theoTerm}
If $\mathcal B$ is any dynamical hierarchy and $m\in\mathbb N$, then there is $i\in \mathbb N$ such that $\G^{\mathcal B}_i(m) = 0$.
\end{MainTheo}

\begin{MainTheo}\label{theoGoodInd}
Let $\odh$ be the greedy dynamical hierarchy of Definition~\ref{defODH}.
Then, $\sf KP$ does not prove that for every $n\in\mathbb N$, there exists $i\in\mathbb N$ with $\goodp ni{\odh} = 0$.
\end{MainTheo}
  
Theorem~\ref{theoGoodInd} should be understood as follows.
Say that an arithmetical formula $\delta$ is {\em bounded} if all quantifiers occurring in $\delta$ are of the form $\forall u < t(\vec v)$ or $\exists u < t(\vec v)$, where $t$ is a polynomial with coefficients in $\mathbb N$, and $\varphi$ is a $\Sigma^0_1$-formula if it is of the form $\exists z \in\mathbb N \ \delta(z)$, where $\delta(z)$ is bounded.
Arithmetical formulas, and $\Sigma^0_1$-formulas in particular, can be encoded in the language of $\sf KP$ in a standard way~\cite{Barwise}.
Since $\goodp xy{\odh}$ is computable~\cite{smorynski1991logical}, there is a $\Sigma^0_1$-formula $\varphi(x,y)$ such that $\varphi(x,y)$ is true on $\mathbb N$ if and only if $\goodp xy{\odh} = 0$.
Then, Theorem~\ref{theoGoodInd} states that $\sf KP$ does not prove $\forall x\in\mathbb N \ \exists y\in\mathbb N  \ \varphi(x,y)$.

\section{Properties of the Upgrade Operator}

The proofs of our main results rely heavily on various regularity properties of the upgrade operator and the deep base change, which we establish in this section.
We begin with what is perhaps their most fundamental property: monotonicity.

\begin{lemma}\label{lemmMonUg}
Let all transformations be from a base hierarchy $B$ to a good successor $ C$ and $m,n,b,c,d\in\mathbb N$ with $b,c,d\geq 2$.

\begin{enumerate}

\item\label{itMonUgOne} If $m<n$, then $\ug m < \ug n$.

\item\label{itMonUgTwo} If $m<n$ and $ \ug b\leq c$, then $\Chgbases bc m \leq  \Chgbases bc n$.

\item\label{itMonUgThree} If $   c < d$, then $\Chgbases bc n \leq \Chgbases bd n$, and the inequality is strict if $ n\geq b$.

\item\label{itMonUgFour} If  $ b<n$, then $ c < \Chgbases bc  n$, and $ c = \Chgbases bc  b$.

\end{enumerate}
\end{lemma}

\begin{proof}
\pref{itMonUgOne} Note that we may assume without loss of generality that $m=n-1$, since otherwise we can see by induction on $n$ that $\ug m \leq \ug(n-1) < \ug n$.
The claim is then immediate, as $\ug n$ is explicitly selected larger than $\ug(n-1)$ whenever $\ug n<\infty$.
\medskip

\noindent\pref{itMonUgTwo} As above, assume that $m=n-1$ and proceed by induction on  $ n $.
If $n < b$, we can apply the first claim to obtain $\ug m<\ug n$, and thus
\[\Chgbases bc m = \Chgbases bd m = \ug m <\ug n = \Chgbases bd n.\]
Otherwise, write $n=_b b^ea+r$ and consider the following cases.
\medskip

\begin{Cases}

\item ($r>0$). We use the induction hypothesis on $r$ to see that
\[\Chgbases  bc m= c^{\Chgbases bc e}\ug a +\Chgbases bc (r-1)  < c^{\Chgbases bc e}\ug a +\Chgbases bc  r =\Chgbases  bc n. \]
\medskip

\item ($r=0$). This case splits in two sub-cases.
\begin{Cases}
\medskip

\item ($a>1$). We have that
\begin{align*}
\Chgbases  bc m & = c^{\Chgbases bc e}\ug (a-1) + \Chgbases bc (b^e-1) < c^{\Chgbases bc e}(\ug a-1) +\Chgbases bc  b^e = \Chgbases  bc n, 
\end{align*}
where from $\ug(a-1)<\ug a$ we inferred $\ug(a-1) \leq \ug a-1$ and we applied the induction hypothesis to $b^e$.
\medskip

\item ($a=1$).
In this case, the induction hypothesis applied to $b$, $e$, and $b^{e-1}$, and the assumption that $  \ug b \leq c$ yields
\begin{align*} 
 \Chgbases  bc m  &= c^{\Chgbases bc (e-1)}\ug (b-1) + \Chgbases bc (b^{e-1}-1) \\
&< c^{\Chgbases bc  e-1 }(\ug  b-1)  + \Chgbases bc  b^{  e-1}  \\
&\leq c^{\Chgbases bc  e-1 }(c-1)  + c^{ \Chgbases bc   e-1}  = c^{\Chgbases bc  e } = \Chgbases  bc n.
\end{align*}
\end{Cases} 
\end{Cases}

\noindent \pref{itMonUgThree} If $n<b$ then $\Chgbases bc n =\ug n = \Chgbases bd n$.
Otherwise, write $n=b^ea+r$.
Induction on $e,r$ yields
\[\Chgbases bc n =  c^{\Chgbases bc  e  }\ug a + \Chgbases bc r \leq c^{\Chgbases bd  e  }\ug a + \Chgbases bd r < d^{\Chgbases bd  e  }\ug a + \Chgbases bd r = \Chgbases bd n, \]
where the first inequality is by the induction hypothesis and the second uses $\ug a, \Chgbases bd  e >0$, as can be derived from \pref{itMonUgOne} and \pref{itMonUgTwo}, given that $a,e>0$.\medskip

\noindent\pref{itMonUgFour} As above, we have that $\Chgbases bc n =  c^{\Chgbases bc  e  }\ug a + \Chgbases bc r$.
If $n=b$, then $e=a=1$ and $r=0$, so $\Chgbases bc b =  c$ follows.
If $n>b$, then $e,a\geq 1$ and either $r>0$ or $\max\{e,a\}>1$, hence  $c<\Chgbases bc n $.
\end{proof}

\begin{remark}
Lemma~\ref{lemmMonUg} does not strictly require $C$ to be a good successor of $B$.
In fact, it may even hold in some cases where $B$ and $C$ are not base hierarchies!
However, the added generality will not be useful to us and it simpler to have this as a `blanket' assumption throughout this section.
For instance, in the above proof, this assumption allows us to assume that $\ug^C_B$ is total. 
Only Lemma~\ref{lemmBaseCalc} fully uses the assumption that $C$ is a good successor, but that lemma is indispensable in our proof of termination.
\end{remark}
 
Next, we establish how the upgrade operator interacts with bases.
Bellow, $\Base_B $ is as given by Definition~\ref{defBase}.

\begin{lemma}\label{lemmBasePreserve}
Consider transformations between a base hierarchy $B$ and a good successor $ C$.
Let $n\geq \min B$, $b = \Base_B(n)$, and   $d$ be such that $\ug n =\Chgbases bd n $.

\begin{enumerate}

\item\label{itBasePreserveA} $n\in B$ if and only if $\ug  n\in C$.

\item\label{itBasePreserveD} $  d = \Base_ C({\ug n}) $.

\item\label{itBasePreserveCompare} If $\min B\leq m<n$, $b'=\Base_B(m)$, and $\ug m =\Chgbases {b'}c m$, then $c\leq d$.

\end{enumerate}
\end{lemma}

\begin{proof}
\noindent \pref{itBasePreserveA}{,}\pref{itBasePreserveD}
If $n\in B$, then $n=b$ and $\Chgbases bd n = \Chgbases bd b = d\in C$.
Moreover, $\Base_C(d) =d$, establishing both claims.
If $n\notin B$, then $b<n$.
Lemma~\ref{lemmMonUg}\pref{itMonUgFour} yields
$d <\Chgbases bd n <S_C(d) $.
This means that $\ug n\notin C$, since $\ug n$ is strictly between an element of $C$ and its successor, and moreover $\Base_C(\ug n)=d$ by definition.
\medskip

\noindent\pref{itBasePreserveCompare} 
Towards a contradiction, assume $d < c$.
Then, $S_C(d)\leq c$, which from the definition of the upgrade operator yields
\[\ug n <S_C (n) \leq c \leq \ug m, \]
where the last inequality follows from Lemma~\ref{lemmMonUg}\pref{itMonUgFour}, since $c   \leq\Chgbases {b'}c m = \ug m $.
This contradicts Lemma~\ref{lemmMonUg}\pref{itMonUgOne}.
\end{proof}

The next lemma provides an equivalent alternative to Definition~\ref{defChg}, which will be more convenient on occasion.
When finding suitable $c$ for computing $\ug n= \Chgbases bc n$ for $n\notin B$, we do not need to verify the inequality $\ug (n-1) < \Chgbases bc n$ directly; it suffices to ensure that $\ug b\leq c $.
If instead $n\in B$, we may simply take the successor of $\ug (n-1)$.
Let us make this precise.

\begin{lemma}\label{lemmOtherDef}
Consider transformations between a base hierarchy $B$ and a good successor $ C$.
Let $n\geq \min B$, $b=\Base_B(n)$, and $c$ be such that $\ug n=\Chgbases bc n$, as given by Definition~\ref{defChg}.

\begin{enumerate}

\item\label{itOtherDefBase} If $n\in B$, then $c = \ug n = S_C(\ug(n-1))$.

\item\label{itBasePreserveOtherDef} If $n\notin B$, then $c$ is the least element of $C$ such that $\ug b\leq c$ and $\Chgbases bcn<S_C(c)$.
\end{enumerate}
\end{lemma}

\begin{proof}
If $n\in B$, then $n=b$ and $c$ is the least element of $ C$ such that $\ug(b-1)< \Chgbases bcb <S_C(c)$.
But, given that $\Chgbases bcb=c$, $c = S_C(\ug (b-1))$ satisfies these properties, while no smaller base can satisfy them.
We conclude that $ c=S_C(\ug (b-1))$ and $\ug n=c$.

Otherwise, since $n\notin B$, we have that $\Base_B(n)\leq n-1$, from which it follows that $\Base_B(n-1)=b$. We may assume inductively that $\ug(n-1) = \Chgbases b{d}(n-1)$ for the least ${d}\in C$ such that $\ug b \leq {d} $ and $\Chgbases b{d} (n-1) < S_C( {d})$ (note that this holds even if $n-1=b$).

Let $c'\geq\ug b $ be the least element of $C$ such that  $\Chgbases b{c'}n<S_C(c') $; we must check that $c'=c$.
By Lemma~\ref{lemmBasePreserve}\pref{itBasePreserveCompare} and the induction hypothesis, $\ug b\leq d\leq c$, so we must indeed have $\ug b\leq c$.
Moreover, $\Chgbases bcn<S_C(c)$, so $c'\leq c$ by minimality of $c'$.

In view of Lemma~\ref{lemmMonUg}\pref{itMonUgTwo},
$ \Chgbases b{c'}(n-1)<\Chgbases b{c'}n<S_C(c') $, so $d\leq c'$ by minimality of $d$.
By Lemma~\ref{lemmMonUg}\pref{itMonUgThree}, $ \Chgbases b{d}(n-1)\leq \Chgbases b{c'}(n-1)$, and hence $\ug(n-1) = \Chgbases b{d}(n-1) <  \Chgbases b{c'}n$.
Since also $\Chgbases b{c'}n<S_C(c') $, we conclude that $c'\geq c$ by minimality of $c$, and hence equality holds.
\end{proof}

Let us now derive some additional useful properties of the deep base-change.

\begin{lemma}\label{lemmUgProps}
Let transformations be between a base hierarchy $B$ and a good successor $ C$.
Let $ b\geq 2 $, $c \geq \ug b $, and $n=_b b^ea+r$.
Then:

\begin{enumerate}

\item\label{itUgPropsDiv} $b\mid n$ if and only if $c\mid \Chgbases bc n$.

\item  \label{itUgPropsNF}
$\Chgbases bc n = _c c^{\Chgbases bc e } \ug  a +\Chgbases bc r$.

\item \label{itUgPropsCoeffs} $\coeffs c{\Chgbases bc n } = \ug \coeffs b n$.

\item\label{itUgPropsChanges} If $d  \geq c $, then $\Chgbases b{d} n = \chgbases c{d} \Chgbases bc n$.

\end{enumerate}
\end{lemma}

\begin{proof}
The left-to-right direction of the first item is proven by induction on $n$.
When $n < b$, the claim follows from $\Chgbases bc n = \ug n <\ug b\leq c $.
Otherwise, write $n=_b b^ea+r$, and note that if $b\mid n$, then  $e\geq 1$ and $b\mid r$.
The induction hypothesis yields $c\mid \Chgbases bc r $ and clearly $c\mid c^{\Chgbases bc e }\Chgbases bc a = \Chgbases bc b^ea $, so\footnote{As per Convention~\ref{convOpers}, $\Chgbases bc b^ea = \Chgbases bc (b^ea)$.}
\[c\mid\big ( \Chgbases bc b^ea + \Chgbases bc r\big ) = \Chgbases bc n.\]
The right-to-left version is similar, but uses Lemma~\ref{lemmMonUg}\pref{itMonUgOne} to see that every $c$-digit of $\Chgbases bc n$ is strictly bounded by $c $.

The second item uses Lemma~\ref{lemmMonUg} to check that the order conditions on $c$-decompositions are satisfied: \ref{lemmMonUg}\pref{itMonUgOne} yields $0=\ug 0<\ug a <\ug b\leq c$ and \ref{lemmMonUg}\pref{itMonUgTwo} yields $\Chgbases bc r < \Chgbases bc b^e = c^{\Chgbases bc  e} $.
The third claim is a corollary of the second by an easy induction.

The fourth item also follows from the second, using the fact that $\Chgbases bc n$ yields a $c$-decomposition and comparing the definitions of the left and right sides.
\end{proof}

Aside from the `official' $b$-decomposition $n =b^ex+y$, it is sometimes convenient to write $n=ba+r $ with $r<b$.
As per Convention~\ref{convOpers}, expressions such as $\Chgbases bc ba+r$ should be understood as $(\Chgbases bc (ba))+r$.
The transformations we consider are well behaved with respect to this alternative decomposition.

\begin{lemma}\label{lemmMonUgRem}
Consider transformations between a base hierarchy $B$ and a good successor $ C$.
Let $n \in\mathbb N$ and $b \geq 2$.
Write $n=ba+r$ with $r<b$ (not necessarily in $b$-decomposition), and let $c \in C$.
Then,
\begin{enumerate}

\item\label{itMonUgRemChg} $\Chgbases bc n = \Chgbases bc ab +\ug r$.

\item\label{itMonUgRemUg} $ \ug n =\ug ba+\ug r$.

\item\label{itMonUgRemSucc} If $\min B \nmid (n+1)$, then $\ug (n+1) = \ug n + 1$.

\item\label{itMonUgRemBase} $\Base_C ( \ug n ) = \Base_ C( \ug ba)$.

\end{enumerate}

\end{lemma}

\begin{proof}
\pref{itMonUgRemChg} Proceed by induction on $a$.
The case for $a=0$ is clear.
Otherwise, write $n=_b b^ex+y$ and note that there is $a'$ such that $y=ba'+r$ and thus $ba =_b b^ex +ba'$.
Then,
\begin{align*}
\Chgbases bc n & = \Chgbases bc b^ex + \Chgbases bc ( ba'+r )\stackrel{\text{\sc ih}}  = \Chgbases bc b^ex + \Chgbases bc   ba'+ \ug r \\
& = \Chgbases bc  (b^ex +    ba')+\ug r = \Chgbases bc ba+ \ug r.
\end{align*}

\noindent \pref{itMonUgRemUg}
If $a=0$ or $r=0$, there is nothing to prove, so we assume otherwise.
Definition~\ref{defChg} yields $d,d'\in C$ such that $\ug ba= \Chgbases bd ba $ and $\ug n = \Chgbases b{d'} n $.
The claim will follow from showing that $d=d'$.
Note by Lemma~\ref{lemmBasePreserve}\pref{itBasePreserveCompare} that $d\leq d'$, so it suffices to check that $d$ satisfies the inequalities required according to Lemma~\ref{lemmOtherDef}\pref{itBasePreserveOtherDef} for $n$, namely, $\ug b\leq d$ and $ \Chgbases bd n <S\us C( d) $.

By Lemmas~\ref{lemmOtherDef} and~\ref{lemmMonUg}\pref{itMonUgOne}, $d\geq \ug b>\ug r$.
By the definition of a base hierarchy, every element of $C$ divides it successor, so $d\mid S_ C( \ug ba)$.
By Lemma~\ref{lemmUgProps}\pref{itUgPropsDiv}, $d \mid \Chgbases bd ba$, hence there are no multiples of $d$ in $(\Chgbases bd ba,\Chgbases bd ba +d)$.
It follows from \pref{itMonUgRemChg} that
\[\Chgbases bd n = \Chgbases bd ba +\ug r < \Chgbases bd ba +d \leq S\us C( d).\]
Thus the desired inequalities hold and we conclude that  $\ug n = \Chgbases bd n = \ug ba +\ug r $.
\medskip

\noindent\pref{itMonUgRemSucc} Let $b = \min B$ and write $n+1=ba+r$ as above, noting that $r>0$.
Using \pref{itMonUgRemUg}, $\ug(n+1) = \ug ba+ \ug r= \ug ba+   r $, where the last equality uses $r<\min B$.
By the same token, $\ug n  =   \ug ba+  r -1 $, and the claim follows.
\medskip

\noindent\pref{itMonUgRemBase} We have already seen that $\ug n= \Chgbases bd n$ for $d=\Base_ C( \ug ba)$, so Lemma~\ref{lemmBasePreserve}\pref{itBasePreserveD} yields $\Base_C ( \ug n ) =d = \Base_ C( \ug ba)$.
\end{proof}

The next lemma shows how to compute $\ug n$ when $n$ is a base and $C$ is a good successor.
Note that both the upper and lower base are involved, as given by Definition~\ref{defBase}.

\begin{lemma}\label{lemmBaseCalc}
Let transformations be from $B$ to a good successor $C$ and let $b\in B\setminus\{\min B\}$.
Let $b'$ be the predecessor of $b$.
Then, $\ug b = \ug(b-b' )+d$, where
\[d = \Base_C(\ug(b-1 ) ) = \basep C{ \ug b  }.\]
\end{lemma}

\begin{proof}
Since $b'$ is the predecessor of $b$, $b'=\Base_B(b-1)$.
Writing $b-1= b-b' +r$ with $r = b'-1$ and letting $d = \Base_C(\ug(b-1 ) )$, we have by Lemma~\ref{lemmMonUg}\pref{itMonUgOne} that $\ug r<\ug b'\leq d$, and by Lemma~\ref{lemmBasePreserve}\pref{itBasePreserveD}, $\ug(b-1) =\Chgbases {b'} d (b-1)$.
Since $b'\mid ( b-b')$, Lemma~\ref{lemmUgProps}\pref{itUgPropsDiv} yields $d\mid \ug(b-b')$, Lemma~\ref{lemmMonUgRem}\pref{itMonUgRemUg} yields $\ug(b-1) = \ug(b-b')+ \ug r < \ug(b-b')+ d$, and Lemma~\ref{lemmMonUgRem}\pref{itMonUgRemBase} yields $d=\Base_C(\ug(b-b'))$.
By Lemma~\ref{lemmOtherDef}\pref{itOtherDefBase}, $\ug b  = S_C( \ug(b-1) )$, and since $C$ is a good successor for $B$, this must be the least multiple of $d$ above $\ug(b-1)$, i.e., $\ug b' = \ug(b-b')+ d$.
Since $b-1$ is not a base, neither is $\ug(b-1)$ by Lemma~\ref{lemmBasePreserve}\pref{itBasePreserveA}, hence there are no bases in $[\ug(b-1),\ug b)$, and thus $\basep C{\ug b} = d$.
\end{proof}

Finally, we note that the upgrade operator is `continuous' in the sense that the computation of $\ug_B^C  n$ only depends on bases below $\ug_B^C  n$ itself.
This can be made precise by restricting $C$ to such values.
For $n\in\mathbb N$ and a base hierarchy $B$, let $B{\upharpoonright_n} = B\cap [0, n]$.

 \begin{lemma}\label{lemmRestrict}
Let $B$, $B'$, $C$, $C'$ be base hierarchies with $\min B   \leq \min C $.
Suppose that $n\in\mathbb N$ is such that $n\geq \min B$, $B{\upharpoonright_{n}} = B'{\upharpoonright_{n}}$, $\ug_B^C n<\infty$, and $m\geq \ug_B^C n$ is such that $C{\upharpoonright_{m}} = C'{\upharpoonright_{m}}$.
Then:
\begin{enumerate}

\item\label{itRestrictTwo} For all $x \leq n$, $\ug^C_B x = \ug_{B'}^{C'} x $.

\item\label{itRestrictThree} For all $x \in\mathbb N$, $b\leq n $, and $c\geq    \ug b $, $\Chgbases bc^C_B x = \Chgbases bc _{B'}^{C'} x <\infty $.

\end{enumerate}
\end{lemma}

The proof follows a straightforward induction and is left to the reader.
Note that in the second item, there is no restriction on $x$; this is due to the fact that $\Chgbases bc$ only applies the upgrade operator to numbers strictly below $b$.

\section{Termination}\label{secTerm}
 
Our proof of termination follows the same basic pattern as that of Goodstein's original proof: namely, each fractal Goodstein progression is assigned a decreasing sequence of ordinals.
Since any decreasing sequence of ordinals is finite, so is the Goodstein process.

In the case of the classic Goodstein process, one exploits the fact that base-$b$ decompositions work in much the same way when $b$ is transfinite.
To make this precise, let us review some ordinal arithmetic.
We assume some familiarity with ordinal addition, multiplication, and exponentiation (see e.g.~\cite{Pohlers:2009:PTBook} for details) and denote the first infinite ordinal by $\om$.

\begin{definition}\label{demOmNF}
Given an ordinal $\kappa$ and $\xi>0$, there exist unique ordinals $\al,\be,\ga$ with $\be<\kappa$ such that $ \xi=\kappa^\al\be+\ga$ and $\ga<\kappa^\al$; this is the {\em $\kappa$-decomposition of $\xi$,} also known as the {\em Cantor decomposition} when $\kappa=\om$.
We write $\xi=_\kappa \kappa^\al\be+\ga$ to indicate that the right-hand side is the $\kappa$-decomposition of $\xi$.
\end{definition}

The ordinal $\ve(\kappa)$ is defined as the least $\ve>\kappa$ such that $\ve=\om^\ve$.
To see that such an ordinal exists, define $\se \kappa 0  = 1$ and $\se \kappa {i+1}  = \kappa^{\se \kappa i }$.
Then, assuming without loss of generality that $\kappa\geq\om$, we see that $\ve(\kappa) = \sup_{n<\om}\se \kappa i$.
If $\xi=_\kappa \kappa^\al\be+\ga < \ve(\kappa)$, then we may also deduce that $\al<\xi$.

Goodstein's theorem is then proven using the well-foundedness of $\ve(\om)$, more commonly denoted $\ve_0$.
For this, we observe that  $\chgbases bc  n$ is well-defined even when $c$ is infinite, as it relies on operations available for arbitrary ordinals.
In particular, we may take $c=\om$.
The analogue of Lemma~\ref{lemmBCProps} is then seen to hold, and, specifically, the $\om$-base-change has the following crucial properties:
\begin{itemize}

\item {\sc Monotonicity.} If $m<n$ and $b\geq 2$, then $\chgbases b\om m < \chgbases b\om n$.

\item {\sc Preservation.} If $2\leq b < c $, then $\chgbases b\om n = \chgbases c\om \chgbases bc n$.

\end{itemize}
These properties are not terribly hard to check, but see e.g.~\cite{FernandezWWalk} for a modern presentation.
Then, if $\mathcal B$ is the dynamical hierarchy of Example~\ref{exClassic} (which produces the standard Goodstein process), we see that if $\goodp n i{\mathcal B} >0$, 
\begin{align*}
\chgbases{i+2}\om \goodp n i{\mathcal B} & = \chgbases{i+3}\om\chgbases{i+2}{i+3} \goodp n i{\mathcal B} &\text{by preservation,}\\
& > \chgbases{i+3}\om (\chgbases{i+2}{i+3} \goodp n i{\mathcal B} - 1) &\text{by monotonicity,}\\
& = \chgbases{i+3}\om   \goodp n {i+1}{\mathcal B}.
\end{align*}
It follows that the sequence $(\chgbases{i+2}\om \goodp n i{\mathcal B})_{i<I}$ is a decreasing sequence of ordinals and hence $I$ is finite, establishing Goodstein's theorem.

We likewise wish to assign countable ordinals to the elements in a fractal Goodstein process.
The basic principle is the same, but things get somewhat more involved.
Given a base hierarchy $B$, we will assign to each $n\in \mathbb N$ an ordinal $o_B(n)$ satisfying the analogues of preservation and monotonicity.
Let us focus on the latter: here, monotonicity means that $m<n$ implies $o_B(m)<o_B(n)$.
As an example, consider $\mathcal B=\{2,6\}$ as in Example~\ref{exUg} and let $n= 6+2^2$.
If, as before, we replace each base by $\om$, we would obtain
\[o_B(n) = \om+\om^\om = \om^\om = o_B(2^2),\]
which clearly violates monotonicity.
Instead, we will assign to the base $6$ an ordinal which is large enough so that the first summand does not cancel.
If the only information available about $o_B(2^2)$ is that it is countable, one `safe' choice is to replace $6$ by $\Om$, the first uncountable ordinal.
We thus arrive at $O_B(n) = \Om +o_B(2^2)$; note the capital $O$ indicating that this ordinal need not be countable.

To obtain $o_B(n)$, in a second phase, we `collapse' $O_B(n)$ to a countable ordinal.
For this, we use the $\vartheta$ function, introduced by Rathjen~\cite{RathjenFragments,RathjenKruskal}.
This function has the feature that it maps $\ve(\Om)$ (more commonly denoted $\ve_{\Om+1}$) to $\Om$.
It is sensitive to the coefficients of an ordinal and, particularly, to their {\em maximal coefficient,} defined by $\mco 0 = 0$ and $\mco{\xi} = \max  \{\al,\mco\eta,\mco\ga\} $ if $\xi=_\Om \Om^\eta\al+\ga$.

\begin{definition}
The function  $\vartheta \colon\ve _{\Om+1} \to \Om$ is defined recursively by
\[\vartheta(\xi) = \min \left \{  \om^\theta>\mco \xi \col \forall \zeta<\xi (\mco\zeta <\om^\theta \Rightarrow  \vartheta(\zeta) <\om^\theta )\right \}.\]
\end{definition}

It is a non-trivial fact that this function is well-defined and injective and its values are additively indecomposable~\cite{BuchholzOrd}.
It allows us to represent ordinals much larger than $\ve_0$; indeed, any ordinal below $\vartheta(\ve_{\Om+1}) := \sup_{k<\om}\vartheta(\se\Om k)$, known as the {\em Bachmann-Howard ordinal,} is a sum of values of the $\vartheta$ function.

With this, our second attempt at an ordinal assignment is to set $o_B(n)=\vartheta(O_B(n))$.
However, we encounter a new problem.
Using our current assignment, we obtain
\[o_B(2^2) = \vartheta(\Om^\Om) > \vartheta(\Om+o_B(1)) = o_B(6+1).\]
To see how we may remedy this, let us review the key order property of $\vartheta$.

\begin{proposition}[\cite{BuchholzOrd}]\label{propThetaOrder}
If $\xi<\zeta<\ve_{\Om+1}$, then $\vartheta(\xi)<\vartheta(\zeta)$ if and only if $\mco\xi<\vartheta(\zeta)$.
\end{proposition}

Since $\zeta^*<\vartheta(\zeta)$, in particular we have that if $\xi<\zeta$ and $\mco\xi \leq \zeta^*$, then $\vartheta(\xi)<\vartheta(\zeta)$.
The flip side is that the value of $\vartheta(\xi)$ can be very large even if $\xi$ is relatively small, provided that $\xi^*$ is large.

Thus, we may ensure that $o_B(4)<o_B(7)$ by `sneaking' a large coefficient into $O_B(7)$.
For this, we first review the natural sum of ordinals, given by $\al \ns 0 = 0 \ns \al = \al $ and
\[
(\om^\al+\be) \ns (\om^\ga+\de)=
\begin{cases}
 \om^\al+(\be \ns (\om^\ga+\de)) &\text{if $\ga \leq \al$,}\\
\om^\ga+ ((\om^\al+\be) \ns  \de) &\text{otherwise,}
\end{cases}
\]
where $\om^\al+\be$, $\om^\ga+\de$ are Cantor decompositions.
The key property of the natural sum is that it is strictly increasing on both components, i.e.~if $\al\leq\al'$, $\be\leq \be'$,  and  at least one of the two inequalities is strict, then $\al\ns\be < \al'\ns\be'$.

Then, $o_B(7)=\vartheta(O_B(7)\ns \bfun_B(7))$, where $ \bfun_B(7)$ is a countable ordinal large enough to ensure monotonicity of the ordinal assignment, by making the maximal coefficient of $O_B(n)$ be large enough.\footnote{For readers familiar with Buchholz's fundamental sequences for $\vartheta$~\cite{BuchholzOrd}, $o^*_B$ plays a role similar to that of $\vartheta^*$.}
In this particular case, we will have $\bfun_B(7) = o_B(6)$, and, in general, $\bfun_B(n)$ will always be either zero or of the form $o_B(b^*)$ for some $b^*\in B$.
For this to work smoothly, $O_B(b)$ will have a slightly different definition when $b\in B$ (see Definition~\ref{defLittleO} below).

In order to make all of this precise, let us first describe a general operation which will allow us to compute $O_B$.

\begin{definition}\label{defOf}
Given $2\leq b $ and a partial function $f\colon \mathbb N \rightharpoonup \Om$ with $[0,b)\subseteq{\rm dom}(f)$, we define $O = O^b_f\colon \mathbb N\to \ve_{\Om+1}$ by writing $x =_b b^ea+r $ and setting
\[O (x) =
\begin{cases}
f(x) &\text{if $x<b$,}\\
\Om^{O (x) }f(a) +O(r) &\text{if $x\geq b $.}
\end{cases}
\]
\end{definition}

The following property will be useful and is easy to check from the definition.
Below, we say that an ordinal $\al$ is a multiple of $\be$ if there is $\ga$ such that $\al=\be\ga$.

\begin{lemma}\label{lemmMultOm}
Let $b\geq 2$ and $f\colon \mathbb N \rightharpoonup \Om$ a partial function with $[0,b)\subseteq{\rm dom}(f)$ and such that $f(n) = 0$ if and only if $n=0$.
Then, for $n\in\mathbb N$, $ O_f^b (n) $ is a multiple of $\Om$ if and only if $n$ is a multiple of $b$.
\end{lemma}

The operation $O^b_f$ already enjoys a preliminary version of what will be the monotonicity of our ordinal assignment.

\begin{lemma}\label{lemmAcuteO}
Let $b\geq 2$ and $f\colon \mathbb N \rightharpoonup \Om$ a partial function with $[0,b)\subseteq{\rm dom}(f)$.
Then, if $f$ is increasing on $[0,b)$, it follows that $O_f^b$ is increasing on $\mathbb N$.
\end{lemma}

\begin{proof}
Let $O(x) = O_f^b(x)$.
We show by induction that if $n>0 $, then $ O(n-1)< O(n)$.
If $n<b$, this follows by assumption on $f$.
Similarly, if $n=b$, then $ O(n-1)<\Om =  O (n)$.

Otherwise, write $n=_b b^ea+r >b$.
If $r>0$, apply the induction hypothesis to $r$.
If $r=0$, then $n-1=_b b^e(a-1)+ b^e-1$.
By induction, $ O(b^e-1) < O(b^e)$, and $f(a-1)<f(a)$ by assumption, so
\[ \ofun(n-1) = \Om^{ O(e)}f(a-1) +  O(b^e-1) <  \Om^{ O(e)}f(a ) = O(n).\]
If $r=0$ and $a=1$, note that $e>1$, since $n>b$.
Then, $b^e-1 =_b b^{e-1}(b-1) + b^{e-1}-1$ and thus $ O(b^e-1) = \Om^{ O(e-1)}f(b-1) + O( b^{e-1}-1)$.
The induction hypothesis yields $O(e-1)< O(e)$ and $ O(b^{e-1}-1)< O(b^{e-1} )$, so since $ f(b-1)<\Om$, $   \ofun(b^e-1)< \ofun(b^e) $.
\end{proof}

With this, we can define the ordinal interpretation we will use to prove termination of the Goodstein process.
For this particular interpretation, the lower base $\baseB n$ of a number $n$ will be more suitable than the upper base (see Definition~\ref{defBase}).
The reason for this is that when $n\in B$, $n$ must `remember' previous ordinals in order to provide the right value for $\bfun_B$, so it is convenient for $n$ to share a base with $n-1$.
We recall that, for non-basic elements, the upper and lower bases coincide.

\begin{definition}\label{defLittleO}
Given a base hierarchy $B$, we define $o_B \colon\mathbb N\to \Omega$ as follows.
If $n<b$, set $o_B(n)=n$.
Otherwise,
\[ o_B(n) =\vartheta( O_B(n) \ns \bfun_B(n)),\]
where $ O_B(n)$, $\bfun_B(n)$ are defined as follows.

Let $b=\baseB n$.
For $f = o_B\upharpoonright [0,b)$, we define an auxiliary operation given by $\ofun^b_B(x) := \ofun_f^b(x)$ for all $x\in\mathbb N$.
Then, consider the following cases.
\medskip

\begin{Cases}

\item {\rm ($n\notin  B\setminus \{\min B\} $).}
Set $\ofun_B(n) = \ofun^b_B(n)$.

\medskip

\item {\rm ($ n\in  B\setminus \{\min B\}$).} Set
\[\ofun _B(n)   = \ofun^b_B(n-b) + o _B(n-b) .\]

\end{Cases}
\medskip

\noindent In either case, let $b^* <n $ be greatest element of $B$ such that $O_B(b^*)+\Om \geq O_B(n) +\Om $, if such a $b^*$ exists, and set $\bfun_B(n) = o_B(b^*)$.
If no such $b^* $ exists, $\bfun_B(n)=0$.
\end{definition}

Although not strictly necessary, is also convenient to set $O_B(n)=O^b_B(n)=n$ and $\bfun_B(n)=0$ if $n<b$, simply to avoid cluttering lemmas with the assumption that $n\geq \max B$.

\begin{example}\label{exOFun}
Let us compute some values of $o_B$ for $ B = \{2,6\}$.
First, we observe that for $n\leq 2$, $o_B$ is just the identity.
We have that $O_B(2) = \Om$ and thus $o_B(2) = \vartheta(\Om)  $.
Similarly, $o_B(4) = \vartheta(\Om^\Om)$.

Now, $6\in B\setminus\{\min B\}$ and $\baseB 6=2$.
We thus have
\[O_B(6) = O^2_B(4)+o_B(4) = \Om^\Om+ \vartheta(\Om^\Om)\]
and $\bfun_B(6) = 0$, since $O_B(2)+\Om<O_B(6)+\Om$.
It follows that $o_B(6) = \vartheta(\Om^\Om+ \vartheta(\Om^\Om) ) $.

Next, we see that $O_B(7) = O_B(6+1) = \Om+1$, and in this case, $\bfun_B(7) = o_B(6) $, so
\[o_B(7) = \vartheta((\Om+1)\ns o_B(6) ) = \vartheta(\Om+ o_B(6)+1). \]
Note that the natural sum is used here to avoid cancellation of $1$.

The next `interesting' case is $o_B(n)$ for $n=6^6 + 6$.
We have that $O_B(n) = \Om^\Om+\Om$.
Thus when computing $\bfun_B(n) $, we see that it is no longer the case that $O_B(6)+\Om \geq O_B(n)+\Om$.
It is also not the case that $O_B(2)+\Om \geq O_B(n)+\Om$.
Hence, $\bfun_B(n) = 0$, and $o_B(n) = \vartheta(\Om^\Om+\Om )$.
\end{example}

\begin{example}\label{exUgBases}
In order to illustrate the behaviour of bases under $o$, let us consider the base hierarchy $C=\{2,4,8\}$.
These bases are as close as they can be to each other, given that each base must divide its successor by definition.

Since $2 = \min C$, we have that $\basep C2=2$, i.e.~$2$ is its own lower base.
As in the previous example, $O_C(2) = \Om$ and $o_C(2) = \vartheta(\Om)$.
Let us also compute $o_C(3)$: here, we have that $\basep C 3=2$ and $3=_2 2+1$, so $O_C(3) = \Om + o_C(1) = \Om+1$.
Meanwhile, $O_C(2)+\Om = \Om\cdot 2 = O_C(3) + \Om$, so $\bfun_C(3) = o_C(2) = \vartheta(\Om)$.
It follows that
\[o_C(3) = \vartheta(O_C(3)\ns \bfun_C(3)) = \vartheta(\Om+\vartheta(\Om)+1 ).\]

Next, we compute $o_C(4)$.
Since $4\in C\setminus \{\min C\}$, the computation is a bit different from that of $o_C(2)$.
Note that $\basep C 4=2$, so $4$ is {\em not} equal to its lower base.
The definition then yields
$O_C(4) = O^2_C (4-2) + o_C(4-2) $.
Since $\basep C 2=2$, we will have that $O^2_C (2) = O_C(2)$ and thus
\[O_C(4) = O_C (2) + o_C(2) = \Om + \vartheta(\Om).\]

Moreover, $O_C(2)+\Om = \Om 2 = O_C(4) + \Om$, so $\bfun_C(4) = o_C(2) = \vartheta(\Om)$.
We thus obtain
\[o_C(4) =\vartheta(O_C(4) \ns o_C(2)) = \vartheta( \Om+\vartheta(\Om)2). \]
The factor $2$ is not an accident, as removing it yields an ordinal below $o_C(3)$.

Finally, we compute $o_C(8)$.
We have that $\basep C 8 =4$, so the definition gives
$O_C(8) = O^4_C (8-4) + o_C(8-4) $.
Unlike in the computation of $o_C(4)$, $\basep C 4\neq 4$, so $O^4_C (4) \neq O_C(4)$.
Instead, $O^4_C(4) = \Om$ and
$O_C(8) = \Om + o_C( 4)$.
Since $O_C(4) +\Om \geq O_C(8) +\Om $, $\bfun_C(8) = o_C(4)$ and
\[o_C(8) = \vartheta(O_C(8)\ns o_C(4)) = \vartheta(\Om + o_C(4) 2).\]
Observe that, as per Definition~\ref{defBase}, $\Base_C(4) = 4$, so here we are using its upper base.

The takeaway is that the base $4$, which is not $\min C$ but is half of its successor, is implicitly assigned two uncountable ordinals: the standard $O_C(4)$, defined in terms of $\basep C 4$, and $O^4_C(4)$, defined in terms of $\Base_C(4)$ and used for computing $O_C(8)$.
In a sense, both ordinals provide `reasonable' interpretations for $4$: in one, we are viewing $4$ as the last number with base $2$, and in the other, as the first with base $4$.
This gives us a bit of wiggle room, and we selected the first option when defining $o_C(4)$ since it encodes the information needed to compute $\bfun_C(n)$ for $n\geq 8$.
The same situation will occur whenever there is $c\in C\setminus \{\min C\}$ such that $S_C(c) =2c$.

Meanwhile, {\em both} of these ordinals are different from the value we would obtain if we removed $4$ from $C$: namely, $O^2_C(4) = \Om^\Om$.
\end{example}

We have already seen how sometimes it is convenient to write numbers in the form $ba+r$ with $r<b$.
Let us see how the large ordinal interpretation behaves in this case.

\begin{lemma}\label{lemmORemain}
Let $B $ be a base hierarchy and $ n = ba+r $ with $r<b$.
Then,
\[O^b_B(n) = O^b_B(ba) + o_B(r).\]
\end{lemma}

\begin{proof}[Proof sketch]
By induction on $n$.
Write $n = _b  b^ex+y$ and observe that either $y=r$ and we can apply Definition~\ref{defOf} directly, or else $b< y <S_B(b) $, and we can apply the induction hypothesis to $y$.
\end{proof}

Following our proof plan, the next step is to prove that the ordinal interpretation $o_B$ is monotone on $n$.

\begin{proposition}\label{propMonO}
Let $B$ be a base hierarchy and $m<n$. Then,
\begin{enumerate}

\item $o_B(m)<o_B(n)$, and

\item if $\baseB m=\baseB n$, then $O_B(m)<O_B(n)$.

\end{enumerate}
\end{proposition}

\begin{proof}
By induction on $m + n$.
If $n \leq \min B$  this is clear, so we assume otherwise.
Let $b=\baseB n$ and consider the following cases.
\medskip

\begin{Cases}

\item ($\ofun_B(m) +\Om <  \ofun _ B(n)+\Om$).
Write $\ofun_B(m) =\Om\al+\eta$ and $\ofun _ B(n) = \Om \ga+\de$ with $\eta ,\de <\Om$.
Note that
\[\Om\al+\eta \leq  \Om\al + (\eta  \ns \bfun_B( n)) <\Om\al+\Om \leq \Om\ga,\]
so
\[\ofun _ B(m)\leq \ofun_B(m) \ns \bfun_B(m) <\ofun _ B(n) \leq \ofun _ B(n) \ns \bfun_B(n) ,\]
establishing the second claim when $\baseB m= b$.
We moreover note that $\mco{   O_B(m)}$ is of the form $o\us B(m')$ with $m'< m$.
By the induction hypothesis, $o_B(m') < o_B(m)$.
Similarly, $\bfun_B(m) = o_B(b^*)$ for some $b^* <m$ and hence $\bfun_B(m)<o_B(n)$.
Clearly, $( O_B(m) \oplus  o_B(m) )^* \leq o_B(m') \oplus \bfun_B(m)<o_B(n)$, where the last inequality uses the fact that $\vartheta (\xi)$ is always additively indecomposable. 
We conclude that $o_B(m)<  o_B(n) $ by Proposition~\ref{propThetaOrder}.
\medskip

\item ($ \ofun_B(m) +\Om \geq \ofun_B(n) +\Om$).
This case divides into the following.
\medskip

\begin{Cases}
\item ($m \leq b$).
Let $b'$ be the least element of $B$ greater or equal to $m$.
Since $m \leq b' \leq b<n$, by the induction hypothesis applied to the second claim, we see that
$O_B( m ) \leq O_B( b')$, hence $O_B( b') +\Om\geq O_B(n)+\Om $.
It follows by the definition of $\bfun_B(n)$ that $\bfun_B(n) = o_B(b^*) $ for some $b^*$ with $m\leq b'\leq b^* <n$.
Thus, the induction hypothesis yields
\[o_B(m) \leq o_B(b^*) = \bfun_ B(n) < o_ B( n).\]
\medskip

\item ($m > b$).
Note that in this case, we have that $b<m<n\leq S_B(b)$, so $\baseB m= b= \baseB n$.
With this in mind, we consider two further sub-cases.
\medskip

\begin{Cases}
\item ($ n \notin B$).
Lemma~\ref{lemmAcuteO} yields $O_B(m)<O_B(n)$, establishing the second claim.
For the first, write $m=ba+r$
and $n=bc+s$ with $r,s<b$.
If $a<c$, the second claim yields $O_B(m) < O_B(bc )$, and since the latter is a multiple of $\Om$ by Lemma~\ref{lemmMultOm},
\[O_B(m) + \Om \leq O_B(bc ) < O_B(n)+\Om.\]
This contradicts the assumption that $ O_B(m) + \Om \geq   O_B(n)+\Om$, and thus $a=c$.

Since $O_B(m) = O^b_B(m) $, we have by Lemma~\ref{lemmORemain} and the induction hypothesis for the first claim that
\[O_B(m) = O^b_B(ba) + o_B(r) < O^b_B(ba) + o_B(s)\]
and
\[O_B(m)+\Om= O^b_B(ba) + \Om = O_B(n)+\Om,\]
so $\bfun_B(m)=\bfun_B(n)$ and $O_B(m) \ns\bfun_B(m) < O_B(n) \ns \bfun_B(n) $.
Finally,
\begin{align*}
\mco{( O_B(m) \ns \bfun_B(m))} & = \max\{\mco{   O^b_B(ba) },o_B(r) \ns \bfun_B(m) \} \\
&\leq \max\{\mco{    O^b_B(ba) },o_B(s)\ns \bfun_B(n) \} = \mco{(  O_B(n) \ns \bfun_B(n))},
\end{align*}
so
\[o_B(m) = \vartheta( O_B(m) \ns \bfun_B(m) ) < \vartheta(  O_B(n)\ns \bfun_B(n) ) = o_B(n).\]
\medskip

\item ($n\in B$).
We may assume that $m = n-1$, since otherwise we can apply the induction hypothesis to $m < n-1$.
Write $m=(n-b)+r$ with $r=b-1$.
Since $n>b$ and $n$ is a multiple of $b$, we have that $r<b\leq n-b$, so the induction hypothesis yields $o_B(r)<o_B(n-b)$ and thus
\[O_B(n) = O^b_B(n-b)+o_B(n-b) > O^b_B(n-b) +o_B(r) = O_B(m),\]
establishing the second claim.
Moreover,
\[O_B(m)+\Om= O^b_B(n-b)+\Om = O_B(n)+\Om,\]
so $\bfun_B(m) = \bfun_B(n) $,
and thus
\begin{align*}
\mco{(  O_B(m) \ns \bfun_B(m))} & = \max\{\mco{   O^b_B(n-b) },o_B(r)\ns \bfun_B(m) \} \\
& \leq \max\{\mco{  O^b_B(n-b) },o_B(n-b)\ns \bfun_B(n) \}\\
& = \mco{  (O^b_B(n-b) +o_B(n-b)\ns \bfun_B(n) )}\\
& = \mco{  (O _B(n)  \ns \bfun_B(n) )},
\end{align*}
and thus
\[o_B(m) = \vartheta(  O_B(m)\ns \bfun_B(m) ) < \vartheta(  O_B(n )  \ns \bfun_B(n)) = o_B(n).\]
 \end{Cases}
\end{Cases}
\end{Cases}
\end{proof}

One intuition for why $O_B(n)$ is computed differently depending on whether or not $n\in B$ is that, since   $\ug_B^C n = S_C(\ug(n-1))$, which is not `much larger' than $\ug(n-1)$ when $C$ is a good successor of $B$, we should also have $O_B(n)$ be not too much larger than $O_B(n-1)$.
Accordingly, the recipe for the case $n\notin B$ will yield larger values than that for $n\in B$, as made precise by the following.

\begin{lemma}\label{lemmDiffOs}
Let $B$ be a base hierarchy $b\in B $ and $n=S_B(b) <\infty$.
Then,
\[O^b_B (n)\geq O_B(n)+\Om.\]
\end{lemma}

\begin{proof}
Since $n$ is a multiple of $b$ by the definition of a base hierarchy, $O^b_B(n)$ and $O^b_B(n-b)$ are multiples of $\Om$ by Lemma~\ref{lemmMultOm}.
By definition,
\[ O_B(n) = O^b_B(n-b) + o_B(n-b) < O^b_B(n-b) + \Om. \]
Write $O^b_B(n-b) = \Om\al$ and $O^b_B(n) =\Om\be$.
Since $O^b_B(n-b) <O_B^b(n)$ by Lemma~\ref{lemmAcuteO} and Proposition~\ref{propMonO}, we must have $\al<\be$, and thus $\al+1 \leq \be$.
But then,
\[ O_B(n) < O^b_B(n-b) +\Om = \Om(\al+1) \leq \Om\be = O^b_B(n).\qedhere\]
\end{proof}

\begin{example}
Recall that in Example~\ref{exUgBases}, we had $C =\{2,4,8\}$ and showed that $O^4_C(4) = \Om$, $O_C(4) = \Om + \vartheta(\Om) $, and $O^2_C(4) = \Om^\Om$.
Note that $O_C(4) +\Om < O^2_C(4)$, so the inequality in Lemma~\ref{lemmDiffOs} can be strict.
The general picture is that if $b \in C$ and $c=S_C(b)<\infty$, then
\[O^c_C(c) < O_C(c) < O_C(c) +\Om \leq O^b_C(c).\]
\end{example}

The final ingredient in the proof of termination is the preservation of the ordinal interpretation under the upgrade.
Note that if we wish to prove that $o_B(m)= o_C({\ug}  m)$, we will first need to show that $O_B(m)= O_C({\ug}  m)$.
The following technical lemma will be a first approximation to this.

\begin{lemma}\label{lemmOTech}
Let all transformations be from $B$ to a good successor $C$.
Suppose that $b \geq 2$ is such that for all $m<b$, $o_B(m) = o_C(\ug m) $.
Then, for all $n\in\mathbb N$ and all $c\geq \ug b$, $O^b_B(n) =  O^c_C(\Chgbases bc n)$.
\end{lemma}

\begin{proof}
Fix $b,c$ satisfying the assumptions and proceed by induction on $n$.
The claim is trivial for $n < \min B$, so we assume otherwise.
Write $n =_b b^ea+r$.
Since $e,r<n$, the induction hypothesis yields
\[O^b_B(n) = \Om^{O^b_B(e)}o_B(a) + O^b_B(r) = \Om^{ O^c_C(\Chgbases bc e)}o_C(\ug a) + O^c_C(\Chgbases bc r) =  O^c_C(\Chgbases bc n).\qedhere \]
\end{proof}

We remark that the assumption that $C$ is a good successor for $B$ is not used explicitly in Lemma~\ref{lemmOTech}, but it ensures that $\ug$ is well defined.
This lemma can then be used in an inductive argument to show that $O_B$ is preserved by the upgrade.

\begin{lemma}\label{lemmPresBigO}
Consider transformations from $B$ to a good successor $C$.
Let $n \in\mathbb N$ and $b =\baseB n$.
Suppose that for all $m < n $, $o_B(m) = o_C(\ug m)$.
Then, $O_B(n) = O_C(\ug n)$.
\end{lemma}

\begin{proof} 
If $n<\min B$, then $O_B(n) = n =\ug n =O_C(\ug n) $, and $O_B(\min B) = \Om = O_C(\min C)  = O_C(\ug \min B)$.
Thus we may assume $n>\min B$, so that $1< b:=\baseB n<n$.
Consider the following cases.
\medskip
 
\begin{Cases}

\item ($n\notin B$).
We have by Lemma~\ref{lemmBasePreserve}\pref{itBasePreserveD} that for $d = \basep C{\ug n}$, $\ug n = \Chgbases bd n $.
By Lemma~\ref{lemmBasePreserve}\pref{itBasePreserveA}, $\ug n\notin C$, so Lemma~\ref{lemmOTech} yields
\[O_B(n) = O^b_B(n) =  O^d_C(\Chgbases bd n) = O_C(\ug n). \]
\medskip

\item ($n \in B$).
We have that $\ug n=S_C( \ug(n-1) ) \in C$.
Write $ n-1 = ba+r$ with $r = b - 1$ and
\[\ug(n-1) = \Chgbases bd (n-1) = \Chgbases bd ba +\ug r = \ug ba + \ug r  ,\]
where the second and third equalities are by items \pref{itMonUgRemChg} and \pref{itMonUgRemUg} of Lemma~\ref{lemmMonUgRem}, so in particular, $\Chgbases bd ba  = \ug ba$.
By Lemma~\ref{lemmBaseCalc},
$\ug n=  \Chgbases bd ba + d $ and $\basep C{ \ug(n-1)} = \Base_ C ( \ug(n-1)) = d$ (where the first equality uses $n-1\notin B$, since $n\in B$).
Lemma~\ref{lemmOTech} yields
$ O^d_C (\ug ba ) = O^b_B( ba  )$, so
\begin{align*}
O_C( \ug n) & = O^d_C(\ug n -d ) +  o_C(\ug n-d ) &\text{by Definition~\ref{defLittleO},} \\
& = O^d_C( \Chgbases bd ba ) +  o_C( \ug ba ) &\text{since $\ug n-d = \Chgbases bd ba  = \ug ba$,} \\
& = O^b_B( ba  ) +o_B( ba ) &\text{by Lemma~\ref{lemmOTech} and assumption,}\\
& = O_B(n)&\text{by Definition~\ref{defLittleO}.} 
\end{align*}
\end{Cases}
\end{proof}

With this, we may prove that $o$ is preserved under the base change.
Lemma~\ref{lemmPresBigO} can be used inductively to show that $O$ is preserved, and thus the bulk of the remaining work lies in showing that $o^*$ is also preserved.

\begin{proposition}\label{propPresO}
If $C$ is a good successor for $B$, then $o_B(n) = o_C(\upgrade B^C n)$ for all $n\in\mathbb N$.
\end{proposition}

\begin{proof} 
Let all transformations be from $B$ to $C$.
For $n< \min B$, $o_B(n) = o_C(\ug n) = n$, so we may assume that $n\geq \min B$.
Proceeding by induction on $n$, Lemma~\ref{lemmPresBigO} and the induction hypothesis yield $O_B(n) = O_C(\ug n)$.
It remains to show that $\bfun_B(n) = \bfun_C(\ug n)$, from which we can conclude that
\[o_B(n) = \vartheta(O_B(n) \oplus \bfun_B(n)) =  \vartheta(O_C(\ug n) \oplus \bfun_C(\ug n)) =o_C(\ug n) ,\]
as desired.

First, we check that $\bfun_B( n ) \leq \bfun_C({\ug}n )  $.
The case where $\bfun_B( n ) =0$ is immediate, so assume $\bfun_B( n ) = o_B(b^*)$ for some $b^*\in B$.
By Lemma~\ref{lemmPresBigO},
\[ \ofun_C({\ug} b^*) +\Om =\ofun_B(b^*) +\Om  \geq \ofun_B(n)+\Om =  \ofun_C({\ug}n)+\Om.\]
Since $\ug b^*\in C$ by Lemma~\ref{lemmBasePreserve}\pref{itBasePreserveA} and $ \ug b^*< \ug n$ by Lemma~\ref{lemmMonUg}\pref{itMonUgOne}, Definition~\ref{defOf} yields $\bfun_C(\ug n) = o _C(c^*)$ for some $c^* \geq \ug b^*$.
Proposition~\ref{propMonO} and the induction hypothesis then yield $o _B(b^*) = o _C(\ug b^*) \leq  o_C(c^*) $.
We conclude that $\bfun_B(n) \leq \bfun_C(\ug n)$.
\medskip

Next, we check that $\bfun_C({\ug } n) \leq \bfun_B(n)$.
As above, we may assume that $\bfun_C( {\ug } n) =  o_C(c^*)$ for some $c^* \in C$.
Let $b^* $ be greatest so that $b^* \in B  $ and $\upgrade{}b^* \leq c^* $, which exists since $\ug  \min B  = \min C  \leq c^*$.
We claim that $\ug b^* =c^*$.

If not, since $\ug b^* \leq c^*$ by choice of $b^*$, we must have that $\upgrade{}b^* < c^*   < {\ug } n$.
By monotonicity (Lemma~\ref{lemmMonUg}\pref{itMonUgOne}), this implies that $b^*< n$.

\begin{claim}\label{claimBStar}
Let $m = \min \{ n, S_B(b^* ) \}  $ and let
\[
m'=
\begin{cases}
m-1&\text{if $m\in B$,}\\
m&\text{otherwise.}\\
\end{cases}
\]
Then, $\baseB {m'} = \baseB {m} = b^*$ and $\ug  m' \geq c^* $.
\end{claim}

\begin{proof}[Proof of Claim \ref{claimBStar}]
We prove the claim according to the following cases.
\medskip

\begin{CasesA}
\item ($m \notin B$).
We must have that $m=n$, since $S_B(b^* )\in B$.
The definition of $m'$ then yields $m'=m =n$, and thus $ c^* < \ug n = \ug m' $ by choice of $c^*$.
Moreover, $b^*<n\leq S_B(b^*) $, and the last inequality is strict since $S_B(b^*) \in B$.
Since $m'=n$, $\baseB{m'} = \baseB{m} = b^*$ follows.
\medskip

\item ($m \in B$).
Then, $m = S_B(b^*)$ and $m'=m-1\notin B$, since consecutive numbers cannot both be bases, given that every base divides its successor.
We thus have that $b^*\leq m' <m $, but $b^*\in B $, so $b^*\neq m'$ and thus $b^*<m'< m =S_B(b^*)$.
We conclude that $\baseB{m'} = b^* = \baseB m$.
Moreover, since $m\in B$ and $b^*$ was chosen to be the maximum element of $B$ with $\ug b^* \leq c^*$, we must have that $c^* < \ug m $.
By Lemma~\ref{lemmOtherDef}\pref{itOtherDefBase}, $\ug m = S_C(\ug(m-1))$, but $c^*\in C$, so $c^* < S_C(\ug(m-1)) $ can only happen if $\ug (m -1) \geq c^*$.
Thus $\ug m'\geq c^*$, as claimed.\qedhere
\end{CasesA}
\end{proof}

Let $c = \basep C{c^*}$, and note that $ c< c^* $, since $ \min C\leq \ug b^*<c^* $.
We thus have that $ c^* = S_C(c)$.
Moreover, $\ug b^*\in C$ by Lemma~\ref{lemmBasePreserve}\pref{itBasePreserveA}, so $\ug b^*<c^*$ yields $\ug b^*\leq c$.

We claim that $\Chgbases {b^*}c  m' \geq c^* $.
Otherwise, $\Chgbases {b^*}c  m' < c^* =S_C(c) $ and $\ug b^*\leq c$, and since $m'\notin B$, Lemma~\ref{lemmOtherDef}\pref{itBasePreserveOtherDef} implies that $\ug m' = \Chgbases {b^*}d  m'$ for the least $d\leq c$ with $\ug b^* \leq d $ and $\Chgbases {b^*}d  m' <S_C(d )$.
Since $c$ has these properties, $d\leq c$.
Lemma~\ref{lemmMonUg}\pref{itMonUgThree} would then yield $\Chgbases {b^*}d  m' \leq \Chgbases {b^*}c  m' <c_*$ and thus $\ug m' < c^*$, contrary to Claim~\ref{claimBStar}.
We conclude that $\Chgbases {b^*}c  m' \geq c^* $, as desired.

We then have that
\begin{align*}
O_C(\ug m ) & = O_B(m) &\text{by Lemma~\ref{lemmPresBigO},}\\
& \geq O_{B}( m ' ) &\text{by Proposition~\ref{propMonO} and Claim~\ref{claimBStar},}\\
& = O^{b^*}_B ( m'  ) &\text{since $m'\notin B$ and $b^*=\baseB{m'}$,}\\
& =  O^c_C (\Chgbases {b^*}c  m')&\text{by Lemma~\ref{lemmOTech},}\\
& \geq  O^c_C( c^* )&\text{by Lemma~\ref{lemmAcuteO} and $\Chgbases {b^*}c  m' \geq c^* $,}\\
& \geq O_C(c^*) + \Om&\text{by Lemma~\ref{lemmDiffOs} and $c^*=S_C(c)$.}
\end{align*}

By definition of $m$, either $m = S_B(b^*)<n$ or $m=n$.
If $m = S_B(b^*) < n$, we have that $\ug m <\ug n$ by Proposition~\ref{propMonO} and $\ug m \in C$ by Lemma~\ref{lemmBasePreserve}\pref{itBasePreserveA}.
Moreover,
\[O_C(\ug m ) \geq O_C(c^*) + \Om \geq O_C(\ug n), \]
where the second inequality is by choice of $c^*$.
But $O_C(\ug m ) \geq  O_C(\ug n)$ implies that $O_C(\ug m ) +\Om \geq  O_C(\ug n) +\Om$, and thus $c^* \geq \ug m $ by definition of $\bfun_C$.
But Lemma~\ref{lemmMonUg}\pref{itMonUgOne} and Claim~\ref{claimBStar} yield $\ug m>\ug(m-1)= \ug m'\geq c^*$, a contradiction.

If $ m = n$, then
\[O_C(\ug n) +\Om > O_C(\ug n) \geq O_C(c^*)+\Om ,\]
again contradicting our choice of $c^* $, since $o^*_C(n) = o_C(c^*) $ requires $O_C(\ug n) +\Om \leq O_C(c^*)+\Om $.
\medskip

In all cases, we obtain a contradiction from the assumption $\ug b^* < c^*$.
Since $\ug b^* \leq c^*$, we thus conclude that $\ug b^*= c^*$.
By Lemma~\ref{lemmPresBigO},
\[O_B(b^*) +\Om   = O_C( \ug b^*) +\Om \geq O_C(\ug n)+\Om = O_B(n)+\Om. \]
Since $b^*\in B$ and $b^*<n$, Definition~\ref{defOf} yields $\bfun_B(n) = o _B(b^+)$ for some $b^+\geq b^*$, and Proposition~\ref{propMonO} and the induction hypothesis then yield $o _B(b^+) \geq o _B(b^*) = o_C(c^*)$.
We conclude that $\bfun_B(n)\geq \bfun_C(\ug n)$, and thus equality holds.
\end{proof}

We now have all the tools to prove that fractal Goodstein processes are always terminating.

\begin{proof}[Proof of Theorem~\ref{theoTerm}]
Let $o_i$ denote $o_{\mathcal B_i}$ and $\upgrade i$ denote $\upgrade{\mathcal B_i}^{\mathcal B_{i+1}}$.
Let $I\leq \infty$ be the length of the Goodstein sequence starting on $m$.
Then, if $ \G^{\mathcal B}_i(m)>0$, we have that
\[ o_{i+1}(\G^{\mathcal B}_{i+1}(m)) = o_{i+1}(\upgrade i \G^{\mathcal B}_i(m)-1) < o_{i+1}(\upgrade i \G^{\mathcal B}_i(m) ) = o_{i }( \G^{\mathcal B}_i(m) ) , \]
where the first inequality is by Proposition~\ref{propMonO} and the second by Proposition~\ref{propPresO}.
It follows that $(o_i(\G^{\mathcal B}_i(m)))_{i<I}$ is a decreasing sequence of ordinals, hence it is finite, and we must have $o_{I-1}(\G^{\mathcal B}_{I-1}(m)) = 0$.
\end{proof}

\section{The Greedy Dynamical Hierarchy}\label{secOuro}

We have shown that fractal Goodstein processes terminate, regardless of $\mathcal B$; in view of Example~\ref{exClassic}, this can be seen as an alternative proof of Goodstein's original theorem if we set $\mathcal B_i=\{ i+2\}$.
However, for other choices of $\mathcal B$, the proof-theoretic strength of termination is much higher than that of the original theorem.
In order to prove this, we will work with the greedy dynamical hierarchy of Definition~\ref{defODH}.
For the sake of generality, we also consider good successors of the form $B^X$ as given by Definition~\ref{defCanon}, where we defined $\partial B$ to be the set of {\em $B$-critical numbers,} i.e.~the set of nonzero numbers that are multiples of their own base.

 \begin{lemma}\label{lemmEndExt}
Let $B$ be a base hierarchy and $i\in\mathbb N$.
Let $X$ be such that $B\subseteq X\subseteq\partial B$, $C = B ^X $, and for each $n\in\mathbb N$, $C_n = B^X_n $, $c_n = \max C_n$, $\ug_n = \ug_B^{C_n}$, and $\Chgbases \cdot\cdot_n =\Chgbases \cdot\cdot _B^{C_n}$.
Then, for every $n\in\mathbb N$ with $b=\Base_B(n)$:

\begin{enumerate}

\item\label{itEndExtTwo} If $c\in C_n\setminus C_{n-1}$, then $c>\max C_{n-1}$.

\item\label{itEndExtThree} If $x<n$, then $\ug_{n} x = \ug_{n-1} x < \infty$.

\item\label{itEndExtFour}  $\ug_n n = \Chgbases b{c_n}_n n<\infty$.

\item\label{itEndExtOne} If $n\geq \min B$, then $C_n$ is  a good successor for $B{\upharpoonright} _n$.

\end{enumerate}
 
\end{lemma}

\begin{proof}
Proceed by induction on $n$.
In this proof, $\ug=\ug_{n-1}$ and $\Chgbases\cdot\cdot = \Chgbases\cdot\cdot_{n-1}$; note that $\ug x$ for $x<n$ and $\Chgbases xyz$ when $2\leq x\leq n$ are finite by the induction hypothesis applied to \pref{itEndExtOne} and Lemma~\ref{lemmRestrict}.

If $n = \min B$, then \pref{itEndExtOne} is easy to check, since $C_n = \{\min B+1\} $ and $\min B+1>\min B \geq 2$, so $C_n$ is a good successor for $B{\upharpoonright}_{n}$.
The other items are either also easy to check or hold vacuously.
Similarly, if $n\notin X$, then $C_n = C_{n-1} $ and all claims follow directly from the induction hypothesis, so we may assume $n>\min B$ and $n\in X$, so that $C_n=C_{n-1}\cup \{c_n\}$, where $c_n$ is the least multiple of $c_{n-1}$ above $\ug (n-1)$.
Let $b'=\Base_B(n-1)$.
\smallskip

\noindent \pref{itEndExtTwo}
By the induction hypothesis applied to \pref{itEndExtFour} and the choice of $c_n$, $\Chgbases {b'}{c_{n-1}} (n-1) = \ug_{n-1}(n-1)   < c_n $, and since $\Chgbases {b'}{c_{n-1}}(n-1) \geq c_{n-1}$ by Lemma~\ref{lemmMonUg}\pref{itMonUgFour}, we have that $c_n > c_{n-1}$.
\smallskip

\noindent \pref{itEndExtThree}
If $x<n$, then by Lemma~\ref{lemmMonUg}\pref{itMonUgOne}, $\ug x\leq \ug(n-1) <c_n $, and since $C_{n-1} = C_n\upharpoonright_{c_n-1}$, \pref{itEndExtThree} follows from Lemma~\ref{lemmRestrict}\pref{itRestrictTwo}.\smallskip

\noindent \pref{itEndExtFour}
According to Definition~\ref{defChg}, $\ug_n n = \Chgbases b{d} n$, where $d\in C$ is least such that $\ug(n-1) < \Chgbases b{d} n <S_{C_n}(d)$, if such a ${d}$ exists.
We claim that $  d =c_n $.
If $n\in B$, this follows from Lemma~\ref{lemmOtherDef}\pref{itOtherDefBase}, given that $c_n=S_{C_n}(\ug(n-1)) $.
If $n\notin B$, since $\ug(n-1) < {c_n}  < \Chgbases b{c_n } n$ and $S_{C_n}(c_n)=\infty$, it suffices to check that no ${d} < {c_n} $ with $d\in  C_n$ is such that $\ug (n-1) < \Chgbases b{{d}} n <  S_C({d})$.
The induction hypothesis applied to~\pref{itEndExtOne} yields $\ug (n-1) = \Chgbases {b }{c_{n-1}} (n-1) <\infty$.
By Lemma~\ref{lemmBasePreserve}\pref{itBasePreserveCompare}, ${d} \geq c_{n-1}$.
Thus, ${d} \in \{c_{n-1},c_n\}$.
So it remains to rule out $d = c_{n-1}$.
However, since $n\in\partial B$, Lemma~\ref{lemmUgProps}\pref{itUgPropsDiv} yields $c_{n-1}\mid \Chgbases b{c_{n-1}} n  $.
By the way we chose $c_n$, we conclude that  $ c_n \leq \Chgbases b{c_{n-1}} n $, which rules out $d=c_{n-1}$, as this requires $\Chgbases b{c_{n-1}} n < S_{C_n}(c_{n-1}) $.
We conclude that $\ug_n n = \Chgbases b{c_n} n$.
\smallskip

\noindent \pref{itEndExtOne}
Since we chose $c_n$ to be the least multiple of  $c_{n-1}$ above $\ug(n-1)$ and $C_{n-1}$ is a good successor for $B$ (and hence a base hierarchy) by the induction hypothesis, one readily checks using the previous items that every element of $C_n$ divides its successor, $\ug^{C_n}_{B{\upharpoonright}_n}$ is total, and bases are mapped to the next available multiple of the previous base; in other words, that $C_n$ is a good successor of $B{\upharpoonright}_n$.
\end{proof}

With this, we are ready to prove that indeed $B\mapsto B^X$ produces good successors.

\begin{prop}\label{propMonC}
Let $B$ be a base hierarchy, $B\subseteq X\subseteq\partial B$, $C=B^X$, and transformations be from $B$ to $C$.
\begin{enumerate}

\item \label{itMonCOne} $C$ is a  good successor for $B$.
 
\item\label{itMonCTwo} If $n\in X \setminus \{\min B\}$, then $\Base_{C} (\ug(n-1)) < \Base_{C} (\ug n)$.
 
\end{enumerate}
In particular, the above items hold for $\partial B$ in place of $X$ and $B^+$ in place of $C$, and the greedy dynamical hierarchy $\odh$ of Definition~\ref{defODH} is a dynamical hierarchy.
\end{prop}

\begin{proof}
For $n\in\mathbb N$, let $C_n=B^X_n$, $\ug_n=\ug^{C_n}_B$, and $\Chgbases\cdot\cdot_n =\Chgbases\cdot\cdot^{C_n}_B$.
\medskip

\noindent\pref{itMonCOne} $C$ is a base hierarchy, because $\min C = \min B+1>2$, and if $c\in C$ and $S_C(c)<\infty$, then $c,S_C(c)\in C_n$ for large enough $n$, so $c\mid S_C(c)$ by Lemma~\ref{lemmEndExt}\pref{itEndExtOne}.

To see that $C$ is a good successor for $B$, let $n\in\mathbb N$ and let $m$ be large enough so that $S_C(n)\in C_m$.
We have that $C{\upharpoonright}_{ S_C(n)} = C_m {\upharpoonright}_{ S_C(n)}$ in view of Lemma~\ref{lemmEndExt}\pref{itEndExtTwo} (as no bases smaller than $S_C(n)$ are added later), so Lemma~\ref{lemmRestrict}\pref{itRestrictTwo} yields $\ug n=\ug_m n<\infty$ and $\ug n$ is always finite.
If moreover $n$ is a base, then $\ug_m n$ is the least multiple of $\Base_B(n-1)$ above $\ug_m(n-1) =\ug(n-1)$, so $C$ satisfies the required condition on bases.
We conclude that $C$ is a good successor for $B$.
\medskip

\noindent\pref{itMonCTwo} If $n\in X$, then by Lemma~\ref{lemmEndExt}\pref{itEndExtFour}, 
\[\Base_C(\ug(n-1)) = \max C_{n-1} <\max C_n = \Base_C(\ug n ) . \qedhere\]
\end{proof}

It then follows from Theorem~\ref{theoTerm} that the fractal Goodstein process for $\odh$ is always terminating; however, this fact is independent of $\sf KP$.
In order to prove this, we need to review fundamental sequences for the Bachmann-Howard ordinal.

\section{Fundamental Sequences}

In most, if not all, past studies on Goodstein processes, one ordinal interpretation is used to prove both termination and independence.
However, in our case, the ordinal interpretation $o_B$ only provides an upper bound on the order-types of numbers in a Goodstein sequence; it is not `tight' enough to provide a lower bound.
While it is possible to find a tight ordinal interpretation, it becomes rather cumbersome.
Instead, our lower bound will be based on the $\psi$-function, which grows much more slowly than $\vartheta$ (see e.g.~\cite{BuchholzSurvey}).

\begin{definition}
Define
$\psi\colon\ve \us{\Om+1} \to \Om$ inductively by $\psi(0)=1$ and, for $\xi>0$,
\[\psi(\xi) = \sup\{\psi(\zeta) +1 :\zeta<\xi\text{ and } \mco\zeta< \psi(\zeta) \}.\]
\end{definition}

From the definition, one obtains $\psi(n) = n+1$ for $n<\om$ and $\psi(\om) = \om$.
Since $\mco \om = \om=\psi(\om)$, $\psi(\om)$ does not affect the computation of $ \psi(\om+1)$, and thus we also have $\psi(\om+1) = \om$.
Similarly, we have that $\psi(\al) = \om$ whenever $\om\leq\al<\Om$, and thus $\psi(\Om) = \om$ as well.
In general, $\psi$ is non-decreasing, but not strictly increasing.
However, we obtain an increasing function if we restrict the domain to ordinals in $\psi$-normal form.

\begin{definition}\label{defPsiNF}
Given $\xi<\ve_{\Om+1}$, we say that $\psi(\xi)$ is in {\em $\psi$-normal form} if $\mco\xi<\psi(\xi)$.
\end{definition}

It will be useful to note that $\psi(\se \Om n)$ is always in $\psi$-normal form and that if $\xi+1$ is in $\psi$-normal form, then $\psi(\xi+1) = \psi(\xi)+1$, which are not hard to check from the definition.
We remark that $\psi$-normal forms as we have defined them treat $\xi$ as an ordinal, rather than a term; in the latter case, we would also require sub-terms to be in $\psi$-normal form.

Despite their different growth rates,  $\psi$ and $\vartheta$ both provide notations for the Bachmann-Howard ordinal, defined equivalently as $\sup_{n<\om}\psi(\se \Om n)$ or $\sup_{n<\om}\vartheta(\se \Om n)$.
Recall that $\vartheta$ only took additively indecomposable ordinals as values; this was helpful to ease computations, but leaves many `gaps.'
In contrast,  the range of $\psi$ is all of $\psi(\ve_{\Om+1}) $, so we can represent any relevant ordinal in terms of $0$, $\psi$, and the function $(\al,\be,\ga)\mapsto \Om^\al\be+\ga$.

In order to connect ordinal notation systems to $\Sigma^0_1$ independence, we may equip them with {\em fundamental sequences,} i.e.~families of sequences $(\xi[n])_{n\in\mathbb N}$, such that when $\xi$ is a limit, $\xi[n]\to \xi$.
Such sequences exist for countable limits, as well as some uncountable ordinals with countable cofinality.

\begin{definition}\label{defTau}
Let $\xi<\ve_{\Om+1}$.
The {\em cofinality} of $\xi$, denoted $\tau(\xi)$, is given recursively by
\begin{enumerate}

\item $\tau(\xi) = \om $ if $\xi<\Om$,

\item $\tau (\Om^\al \be+\ga) = \tau  (\ga)$ if $0<\ga<\Om^\al$,

\item $\tau (\Om^\al \be ) = \om$ if $\be<\Om$ is a limit,

\item $\tau (\Om^\al (\be+1) ) = \tau(\al)$ if $\al$ is a limit, and

\item $\tau(\Om^ {\al+1}(\be+1)) = \Om$.

\end{enumerate}
\end{definition}

Note that successors are assigned cofinality $\om$, so this is not a `true' cofinality, but this definition is convenient to uniformly define fundamental sequences.
It is convenient to also define fundamental sequences for ordinals $\xi$ of cofinality $\Om$, but in this case, $\xi[\theta]$ must be defined for all $\theta<\Om$.

\begin{definition}
Let
\[\psi[\ve_{\Om+1}] = \{\xi <\ve_{\Om+1} : \mco\xi< \psi(\ve_{\Om+1})\} .\]
Then, for $\xi \in \psi[\ve_{\Om+1}]$ and $\theta<\tau(\xi)$, we define $\xi[\theta]$ inductively by
\begin{multicols}2
\begin{enumerate}

\item $ \fs 0\theta = \fs 1 \theta = 0$,

\item $ \fs{( \Om^\alpha\be + \ga) }{\theta}  =  \Om^\alpha\be + \fs{ \ga }{\theta}  $ if $\ga>0$,

\item $ \fs{ (\Om^\alpha\be)   }{\theta}  =  \Om^\alpha(\fs\be\theta) $ if $\be$ is a limit,

\item $ \fs{ (\Om^\alpha(\be+1))}{\theta}  =  \Om^\al \be +  {\fs {\Om^\al}\theta  } $ if $\be>0$,

\item $ \fs{ \Om^{\al+1}   }{\theta}  =  \Om^\al \theta $,

\item $ \fs{ \Om^\alpha }{\theta}  =  \Om^{\fs \al\theta  } $ if $\al$ is a limit,

\item $\fs{\psi(\zeta)}{\theta}=\psi(\fs\zeta\theta) $ if $\tau(\zeta) = \om$, 

\item if $\tau(\zeta)=\Om$ and $\xi=\psi(\zeta)$, then $\fs{\xi}0 = 0$
and
$\fs{\xi}{\theta+1} = \psi(\fs\zeta{\fs{\xi}\theta })$,

\end{enumerate}
\end{multicols}
\noindent where all expressions are written in either $\Om$-decomposition or $\psi$-normal form, accordingly.
\end{definition}

As a useful example, it is readily checked that $\om[i] = i$ for all $i$, using the fact that $\om=\psi(\Om)$.
These fundamental sequences converge and preserve $\psi$-normal forms (see e.g.~\cite{WeiermannInvestigations}).

\begin{lemma}\label{lemmFS}
If $\la \in \psi[\ve_{\Om+1}]$ is a limit, we have that $(\la[\iota])_{\iota<\tau(\la)}$ is an increasing sequence which converges to $\la$.

Moreover, if $\psi(\zeta)$ is in $\psi$-normal form and $i<\om$, then:
\begin{enumerate}

\item\label{itFSSucc} $\psi(\zeta)$ is a successor if and only if $\zeta$ is a successor, in which case $\psi(\zeta) =  \psi(\zeta[i])+1$ and $\zeta = \zeta[i]+1 $.

\item\label{itFSCount} If $\tau(\zeta) =\om$, then  $ \psi(\fs\zeta i) $ is in $\psi$-normal form.

\item\label{itFSUn} If $\tau(\zeta)=\Om$ and $\xi=\psi(\zeta)$, then $\psi(\fs\zeta{\fs{\xi}i })$ is in $\psi$-normal form.

\end{enumerate}
\end{lemma}

We also note that fundamental sequences `almost always' map uncountable ordinals to uncountable ones.

\begin{lemma}\label{lemmZetaIota}
Let $\zeta \in \psi[\ve_{\Om+1}] $ and $\iota<\tau(\zeta)$.
\begin{enumerate}

\item \label{itZetaIiotaUnc}

If $\zeta >\Om$, then either $ \zeta[\iota]\in \{0,1 \} $ and $\iota=0$, or else $\zeta[\iota]$ is uncountable.

\item \label{itZetaIiotaC} If $\zeta$ is a limit and $\iota=\om\la+\ell$ with $\ell<\om$, then  $\zeta[\iota]$ is of the form $\om \rho + r$ with $r\leq \max\{2,\ell \}$.

\end{enumerate}
\end{lemma}

\begin{proof}
\pref{itZetaIiotaUnc} Write $\zeta=\Om^\eta\al+\ga$ and proceed by induction on $\zeta$.
Since $\zeta > \Om$, $\zeta[\iota]$ is countable only if $\ga=0$.
If $\al $ is a limit, then  $\zeta[\iota]=\Om^\eta(\al[\iota])$, which is either uncountable or possibly zero with $\iota=0$.
If $\al$ is a successor, then $\zeta[\iota]$ is uncountable.

Otherwise, $\al=1$.
If $\eta$ is a limit, then $ \zeta[\iota ]$ can only be countable when $\eta[\iota ] = 0$, whence $\iota=0$ and $\zeta[\iota] =  \Om^0  = 1 $.
Otherwise, $\eta =\delta+1$ is a successor and $\delta>0$ since $\zeta>\Om$, so $\zeta[\iota]$ is uncountable or zero when $\iota=0$.
\medskip

\noindent\pref{itZetaIiotaC}
Proceed by induction on $(\iota,\zeta^*,\zeta)$, ordered lexicographically (i.e., a primary induction on $\iota$, with a secondary induction on $\zeta^*$, and a tertiary induction on $\zeta$).
\smallskip
\begin{Cases}
\item ($\zeta<\Om$). Write $\zeta=\psi(\xi)$ in $\psi$-normal form, so that $\xi^*<\zeta = \zeta^*$.
In this case, $\tau(\zeta)=\om$, so $\iota $ is finite and thus we must have $\iota=\ell$.
Since $\psi(\xi)$ is a limit in $\psi$-normal form, $\xi$ is uncountable; for, if $\xi$ is finite, then  $\psi(\xi) = \xi+1$, which is not a limit, and if $\xi$ is countably infinite, then $\psi(\xi) = \om\leq \xi$, so $\psi(\xi)$ cannot be in $\psi$-normal form.
Consider the following sub-cases.
\medskip

\begin{Cases}
\item ($\xi = \Om$). 
We directly compute $\zeta[\iota] = \iota  $, which is of the desired form since $\iota=\ell$.
\medskip

\item ($\xi > \Om$).
We have that $\zeta[\iota] = \psi(\xi[\iota'])$, where $\iota'$ depends on the cofinality of $\xi$.
By \pref{itZetaIiotaUnc}, either $\xi[\iota']\leq 1$ or $\xi[\iota']$ is uncountable.
In the first case, $\zeta[\iota] \leq \psi(1) = 2$, so we may assume that $\xi[\iota']$ is uncountable.
We consider two further sub-cases, according to the cofinality of $\xi$.
\medskip

\begin{Cases}

\item ($\tau(\xi) =\om$). 
In this case, $\iota'=\iota$.
Since $\xi^*<\zeta=\zeta^*$, we may apply the induction hypothesis to obtain $\xi[\iota] =    \om \rho+r $ with $r\leq\max \{2,\ell\} $ and $\rho>0$ (as we are assuming that $\xi[\iota]$ is uncountable).
Then, using Lemma~\ref{lemmFS}\pref{itFSSucc}, we see that $\zeta[\iota] = \psi(\xi[\iota]) = \psi( \om \lambda+r)  = \psi( \om \lambda )+r$, which is of the desired form since $\psi( \om \lambda )$ is a limit.
\medskip

\item ($\tau(\xi) =\Om$). Either $\iota=0$ and $\zeta[\iota] = 0$, which is of the desired form, or $\zeta[\iota] = \psi(\xi[\iota']) $, where $\iota'=0$ if $\iota=0$ and otherwise $\iota'= \zeta[\iota-1]$.
If $\iota'=0$, we may set $\rho'=r'=0$ and note that $\iota'= \om\rho'+r'$ with  $r' < \max\{2, \ell \}$.
If $\iota'= \zeta[\iota-1]$ the induction hypothesis is available since $\iota-1<\iota$ to obtain $\zeta[\iota-1] = \om\rho'+r'$ with $r'\leq \max\{2,\iota-1\} = \max\{2,\ell -1\}$.
In either case, $\iota' = \om \rho'+r'$ with $r' \leq \max\{2, \ell \}$.
Applying the induction hypothesis once again, this time by $\xi^*<\zeta^*$,
$\xi[\iota'] = \xi[\om\rho'+r']  = \om\rho +r $
for some $r\leq \max\{2,r'\} \leq \max\{2,\ell\}$.
Since we are assuming that $\xi[\iota'] $ is uncountable, $\rho>0$.
Using Lemma~\ref{lemmFS}\pref{itFSSucc}, we see that $\zeta[\iota]=\psi(\xi[\iota'])=\psi(\om\rho +r) = \psi(\om\rho) +r $, which is of the desired form.

\end{Cases}
\end{Cases}
\medskip

\item ($\zeta\geq \Om$). Write $\zeta=\Om^\eta\al+\ga$.
If $\ga>0$, the induction hypothesis applies readily to $\ga$.
Otherwise, $\zeta=\Om^\eta\al$.
We cannot have $\zeta[\iota] $ be a successor if $\al$ is a limit, so we write $\al=\de+1$ and see that $\zeta[\iota] =  \Om^\eta\de + \Om^\eta[\iota]$, which is a successor if and only if $\Om^\eta[\iota]$ is a successor.
If $\eta = \rho+1$, then $\Om^\eta[\iota] = \Om^\rho \iota$, which can only be a successor if $\rho=0$ and $\iota <\om$, in which case we obtain $\zeta[\iota] =  \Om^\eta\de + \iota $, which is of the required form.
If $\eta$ is a limit, $\Om^\eta[\iota] = \Om^{\eta[\iota]}$, which can only be a successor if $\iota=0$ and $\eta[\iota] = 0$, in which case $\zeta[\iota] =  \Om^\eta\de + 1$, again of the desired form.\qedhere
\end{Cases}
\end{proof}






 

These fundamental sequences can be used to define very fast-growing functions, such as the function $F$ below.

\begin{definition}\label{defFFun}
For $i<\om$ and $\al<\psi(\ve_{\Om+1})$, define $\fsi\al i$ recursively by $\fsi\al 0 = \al $ and $\fsi\al {i+1} = \fs {\fsi \al i}{i+1} $.
Define $F(n)$ to be the least $\ell$ such that
$\fsi{\psi(\se \Om n)} \ell = 0$.
\end{definition}

The sequence $(\fsi\xi i )_{i<I}$ `steps down the fundamental sequences.'
The function $F$ is total since $ \fsi{\psi(\se \Om n)} {i+1} < \fsi{\psi(\se \Om n)} i$ whenever the right-hand side is not zero, but totality is not provable in $\sf KP$.
More generally, $F$ majorizes all provably total functions of $\sf KP$~\cite{buchholz1988proof,Eguchi}.

\begin{theorem}\label{theoKPInc}
Let $f$ be a computable function and suppose that ${\sf KP}\vdash \forall x\exists y (y=f(x)) $.
Then, $\exists m \ \forall n>m \ \big ( f (n) < F(n) \big)$.
\end{theorem}

Here, the expression $y=f(x)$ should be represented according to the conventions of Theorem~\ref{theoGoodInd}.

\begin{remark}
Standard proof-theoretic analyses are based on variants of $\psi$ which take only additively indecomposable values (or even strongly critical ones), so there is some verification involved in applying Theorem~\ref{theoKPInc} to our `slow'  $\psi$.
Using the notation of Buchholz~\cite{BuchholzSurvey}, let $\psi^\mathbb H$ denote the collapsing function which takes additively indecomposable values.
According to Theorem 1 of Cichon et al.~\cite{BCW}, we only require that $\psi(\ve_{\Om+1}) = \psi^ \mathbb H(\ve_{\Om+1})$ and that our fundamental sequences satisfy certain regularity properties.
The desired equality can be verified using the identity $\psi^ \mathbb H(\alpha) = \psi(\Omega\alpha)$ and the proofs of regularity in e.g.~\cite{WeiermannInvestigations} specialize to our more restrictive $\psi$.
Lemma~\ref{lemmFS} is an example of this specialization.
\end{remark}

\section{Independence}

Given Theorem~\ref{theoKPInc}, in order to to prove independence of Theorem~\ref{theoTerm}, it suffices to compare our Goodstein process to the process of stepping down the fundamental sequences.
As pointed out above, the ordinal interpretation $o_B$ is not suitable here.
Instead, we will introduce a second ordinal interpretation, $u_B$, based on $\psi$.
The issue here is that $\psi$ does not have some of the key properties of $\vartheta$ which allowed us to define a monotone ordinal assignment, and thus $u_B$ will not be monotone.
This makes $u_B$ unsuitable for proving termination of the fractal Goodstein process yet, surprisingly, it will not be a problem for proving independence.

To illustrate this, let us turn to the proof of independence of the original Goodstein principle.
In this setting, one typically shows that, for $b\geq 2$,
\[\chgbases b\om (n-1) \geq  \chgbases b\om  n  [b-1] \]
(where $ \chgbases b\om  n [x]$ should be read as $(\chgbases b\om  n) [x]$).
Using properties of the fundamental sequences and the ordinal assignment, including monotonicity, this suffices for concluding that the Goodstein process takes longer than the process of stepping down the fundamental sequences below $\ve_0$, which is not provably terminating in $\sf PA$.
But there is a different argument which avoids the use of monotonicity of the ordinal assignment.

Suppose that we show that for every $n$ and $b\geq 2$, there is $n'<\chgbases{b}{b+1}n$ such that $\chgbases {b+1}\om n' =  \chgbases b\om n  [b-1]$.
For example, if $b=3$ and $n=3^3$, we could take $n'=4^2$, and observe that
\[\chgbases 4\om n' = \om^2 = \om^\om[2] = \chgbases 3\om n[2]. \]
It is not too hard to check that such an $n'$ can always be found.

Then, reason as follows.
Let $m\in\mathbb N$ and let $\xi=\chgbases 2\om m$.
Inductively define a sequence $(n_i)$ with $n_0=m$, $n_{i+1}<\chgbases{i+2}{i+3}n_i$, and such that $\chgbases{i+3}\om n_{i+1} = \chgbases{i+2}\om n_i[i+1]$.
By induction on $i$, $\chgbases{i+2}\om n_i = \fsi \xi i $.
Moreover, letting $m_i=\goodp mi{\mathcal B}$, we inductively see that
\[  m_{i+1} =  \chgbases{i+2}{i+3}m_i -1 \geq \chgbases{i+2}{i+3}n_i - 1 \geq n_{i+1};\]
note that the first inequality requires monotonicity of the {\em finitary} base change, but not the {\em in}finitary one.
Thus, $m_i>0$ as long as $\fsi \xi i>0$, and we have avoided using the monotonicity of the ordinal assignment.
We only need for the finitary base change to be monotone, but this corresponds to the monotonicity of the upgrade operator, which will still be available.

Thus our function $u_B$ will be such that, letting $u_i=u_{\odh_i}$ and $\ug^{i+j}_i=\ug_{\odh_i}^{\odh_{i+j}}$, for every $n\in\mathbb N$, there is $n'<\ug^{i+j}_i n$ such that $u_{i+j}(n') = u_i(n)[j]$.
We can then replicate a variant of the above argument to show that the Goodstein process majorizes the function $F$ of Definition~\ref{defFFun}.
The function $u_B$ follows a similar pattern to $o_B$, where a `large' ordinal assignment $U_B$ is first defined, and then `collapsed' to $u_B$.

\begin{definition}
Given a base hierarchy $B$ and $n\in\mathbb N$, define $u_B(n) <\psi(\ve_{\Om+1})$ inductively on $ n$.
If $n<\min B$, then $u_B(n) = n$.

Otherwise, let $b=\Base_B(n)$.
In this case, we also define $U_B(n) \in \psi[\ve_{\Om+1}]$.
Let $f = u_B\upharpoonright [0,b)$ and define $U^b_B(x) = O_f^b(x) $, where $O_f^b$ is as in Definition~\ref{defOf}, and set $U_B(n)  = U_B^b(n)$.
Then, set $u_B(n)  = \psi(U_B(n))$.
\end{definition}

We can deduce some basic properties of $n$ from $U_B(n)$.

\begin{lemma}\label{lemmULimit}
Let $B$ be a base hierarchy and $n \geq \min B$.
\begin{enumerate}

\item\label{itULimitMin} The following are equivalent: $u_B(n)$ is a limit, $U_B(n)$ is a limit, and\linebreak $\min B\mid n$.

\item\label{itULimitBase} If $\tau(U^b_B(n))=\Om$, then $\Base_B(n)\mid n$.

\end{enumerate}
\end{lemma} 

\begin{proof}
\pref{itULimitMin} The equivalence of $u_B(n)$ and $U_B(n)$ being limits is a consequence of Lemma~\ref{lemmFS}\pref{itFSSucc}, so we check that $U_B(n)$ is a limit if and only if $\min B\mid n$.
Let $b=\Base_B(n)$ and $U=U^b_B$ and write $n=_b b^ea+r$, so that $U_B(n)=\Om^{U(e)}u_B(a)+U(r)$.
Since $\min B\mid b$, we have that if $r=0$, then $\min B\mid n$ and $U_B(n)$ is a limit.
If $0<r<\min B$, then $\min B\nmid n$ and $U(r) =  r$, so $U_B(n)$ is not a limit.
Otherwise, by induction, we have that $U(r)$ is a limit if and only if $\min B\mid r$, hence $U_B(r)$ is a limit if and only if $\min B\mid n$.
\medskip

\noindent\pref{itULimitBase} If $\tau(U_B(n)) = \Om$, then $\Om\mid U_B(n)$, so by Lemma~\ref{lemmMultOm}, $\Base_B(n)\mid n$.\qedhere

\end{proof}

\begin{remark}\label{remOmCof}
Note that \pref{itULimitBase} is not an equivalence, since for example if $b= \Base_B(n)$ and $n=_b b^1a$ with $u_B(a) $ a limit, then $\tau(U_B(n)) = \om$, yet $b \mid n$.
However, we will not need to precisely identify those $n$ with $\tau(U_B(n)) = \Om$, and the condition $\Base_B(n)\mid n$ is a reasonable approximation.
\end{remark}

Despite being simpler to work with in many ways, the main drawback of the $\psi$ function over $\vartheta$ is the need to check whether ordinals are in normal form before performing computations with them.
Likewise, the $u_B$ and $U_B$ functions come with normal forms that allow them to work as intended, despite the lack of monotonicity.

\begin{definition}
Let $B$ be a base hierarchy and $b\in B$.
Given $n\in\mathbb N$ and $b\geq 2$, we simultaneously define $n $ to be in {\em $u_B$-normal form,} {\em $U_B$-normal form,} and {\em $U^b_B$-normal form} inductively on $\max \{n,b\}$ as follows.

\begin{enumerate}

\item If $n<\min B$, then $n$ is in $u_B$-normal form, $U^b_B$-normal form, and $U_B$-normal form.

\item If $n=_b b^ea+r$, then $n$ is in $U^b_B$-normal form if and only if $a$ is in $u_B$-normal form, $e$, $r$ are in $U^b_B$-normal form, and
\[\Om^{U^b_B(e)}u_B(a) +U^b_B(r)\]
is in $\Om$-decomposition.

\item If $b=\Base_B(n)$, $n$ is in $U_B$-normal form if and only if $n$ is in $U^b_B$-normal form, and $n$ is in $u_B$-normal form if and only if $n$ is in $U_B$-normal form and for all $c\in\coeffs b n$, $u_B(c) < u_B(n)$.

\end{enumerate}

\end{definition}

The notion of $u_B$-normal forms is closely related to that of $\psi$-normal forms.

\begin{lemma}\label{lemmMC}
Let $B$ be a base hierarchy, $n\geq \min B$, and $b=\Base_B(n)$.
Then, $n$ is in $u_B$-normal form if and only if $n$ is in $U_B$-normal form and $\psi(  U_B(n))$ is in  $\psi$-normal form.
\end{lemma}

\begin{proof}
Let $b=\Base_B(n)$.
The $U_B$-normal form condition is unchanged from the original definition and is not hard to check that $ \mco{(U_B(n))} = \max u_B[\coeffs bn]$.
Thus, $ \mco{(U_B(n))} < \psi(U_B(n)) $ if and only if for every $c\in \coeffs bn$, $u_B(c) <u_B(n)$.
\end{proof}

\begin{example}\label{exUFun}
Let $B=\{2,6\}$ as in Example~\ref{exOFun} and let us compute some values of $u_B$.
We have that $u_B(0) =  0$  and $u_B(1) =   1$, since $u_B$ is the identity below $\min B$.
We then note that $U_B(2) = \Om$, so $u_B(2) = \psi(\Om) = \om$.
Similarly, $U_B(4) = \Om^\Om$ and $u_B(4) = \psi(\Om^\Om)$.

Next, we see that $U_B(6) = U_B^6(6) = \Om $, so $u_B(6)=\psi(\Om )  =u_B(2)$.
This already shows that $u_B$ is not an injective function.
In fact it is not even non-decreasing, as $u_B(4) >  u_B(6)$.
\end{example}

Although the map $u_B$ is not monotone, we do need it to be preserved under the upgrade operation.
The following lemma will be useful in establishing this.

\begin{lemma}\label{lemmPresU}
Let all transformations be from $B$ to a good successor $C$.
Suppose that $b\in B$ is such that for all $a < b$, $u_B( a ) = u_C(\ug a ) $, and if $a$ is in $u_B$-normal form, then $\ug a$ is in $u_C$-normal form.
Then, for all $n\in\mathbb N$ and $c\geq \ug b$, $U^b_B(n)= U^c_C(\Chgbases bc n)$.
Moreover, if $n$ is in $U^b_B$-normal form, then $\Chgbases bc n$ is in $U^c_C$-normal form.
\end{lemma}

\begin{proof}
Proceed by induction on $n$.
The case $n=0$ is trivial.
Otherwise, write $n =_b b^ea+r$.
Then, $ \Chgbases bc n = _c  c^{\Chgbases bc e} \ug a+{\Chgbases bc r} $.
We have by assumption that $u_C(\ug a) = u_B(a)$.
The induction hypothesis then yields  
\[U^b_B(n) = \Om ^{U^b_B(e)}u_B( a )+ U^b_B(r) = \Om ^{U^c_C(\Chgbases bc e)}u_C(\ug a) + U^c_C(\Chgbases bc r) = U^c_C(\Chgbases bc n). \]

If moreover $n$ is in $U^b_B$-normal form and $y\in \coeffs c{ \Chgbases bc n}$, then by Lemma~\ref{lemmUgProps}\pref{itUgPropsCoeffs}, $y =\ug x$ for some $x\in \coeffs bn$, which by assumption implies that $y$ is in $u_C$-normal form.
Moreover, $e$ and $r$ are in $U^b_B$-normal form, which by the induction hypothesis implies that $ \Chgbases bc  e$ and $\Chgbases bc r$ are in $U^c_C$-normal form.
Finally, $\Om^{U^c_C(\Chgbases bc e) } u_C(\ug a) +U^c_C(\Chgbases bc r)$ is in $\Om$-decomposition since $\Om^{ U^b_C( e )} u_B(  a) +U^b_B(  r)$ is in $\Om$-decomposition and the respective ordinals are equal.
\end{proof}

With this, we show that the upgrade preserves not only the values of $u$, but also preserves $u$-normal forms.

\begin{proposition}\label{propPresU}
With transformations from $B$ to a good successor $ C$, for all $n\in\mathbb N$, $u_B(n)= u_C(\ug n)$ and $U_B(n)= U_C(\ug n)$.
Moreover, if $ n $ is in $u_B$-normal form, then $ \ug n$ is in $u_C$-normal form, and if $ n $ is in $U_B$-normal form, then $ \ug n$ is in $U_C$-normal form.
\end{proposition}

\begin{proof}
Proceed to prove all claims by induction on $n$.
We first prove that $U_B(n)= U_C(\ug n)$.
Let $b=\Base_B(n)$ and $c=\Base_C(\ug n)$.
By Lemma~\ref{lemmBasePreserve}\pref{itBasePreserveD}, 
$\ug n = \Chgbases bc n $.
The induction hypothesis yields that for $x\in\coeffs bn$, $u_B(x) = u_C(\ug x)$, and if $x$ is in $u_B$-normal form, then $\ug x$ is in $u_C$-normal form.
Lemma~\ref{lemmPresU} then yields $U_B(n) = U^b_B(n) = U^c_C(\ug n) =  U_C(\ug n) $.
From the condition $U_B(n) =    U_C(\ug n) $, we then obtain
\[ u_B(n) = \psi( U_B(n)) =  \psi(  U_C(\ug n)) = u_C(\ug n). \]

Now, assume that $n$ is in $u_B$-normal form.
Then, $n$ is in $U^b_B$-normal form, so by Lemma~\ref{lemmPresU}, $\ug n = \Chgbases bc n$ is in $U^c_C $-normal form, which since $c=\Base_C(\ug n)$ implies that $\ug n$ is in $U_C$-normal form.
Let $y \in \coeffs c{\ug n} $.
By Lemma~\ref{lemmUgProps}\pref{itUgPropsCoeffs}, $y=\ug x $ for some $x\in \coeffs bn$.
Since $n$ is in $u_B$-normal form, $x$ is in $u_B$-normal form and $u_B(x)<u_B(n)$.
By the induction hypothesis, $y = \ug x$ is in $u_C$-normal form and
\[u_C(y ) = u_B(x)<u_B(n)=u_C(\ug n),\]
where the first equality is by the induction hypothesis and we have already established the second.
\end{proof}

Next, our goal is to match up the fundamental sequences with values of the new ordinal interpretation.
This is where greedy successors come in.
To elucidate the relation between them and the fundamental sequences, let $B $ be a base hierarchy and $C=B^+$.
Suppose that $B$ can `compute fundamental sequences' for $n$ up to some $i\in\mathbb N$, in the sense that there is $m<n $ such that $ u_B(m) = u_B(n)[i]$.
The claim is that $C$ can then compute fundamental sequences for $\ug n$  up to $i+1$.
The critical case is that where $\zeta = U_B(n)$ has uncountable cofinality.
Then, $ u_B(n)[i+1] = \psi(\zeta [u_B(n)[i] ]) = \psi(\zeta [u_B(m) ])$.
Intuitively, $\zeta [  u_B(m) ]$ is obtained by replacing the `rightmost' occurrence of $\Om$ by $u_B(m) $, e.g.~if $\zeta= \Om^\Om $, then $u_B(n)[i+1] = \psi( \Om^{u_B(m)})$.
If $b=\Base_B(n)$, this would be reflected by replacing the `rightmost' occurrence of $b$ by $m$, so that if $n=b^b$, we would have $u_B(n)[i+1] = u_B(b^{m})$.
But, we may not be able to do this, since we might have that $m>b$.
 
However, the situation changes on $C$.
Since $\zeta$ has uncountable cofinality, it must be a multiple of $\Om$, so $n$ is a multiple of $b$ by Lemma~\ref{lemmULimit}\pref{itULimitBase}; in other words, $n\in\partial B$.
Using Proposition~\ref{propMonC}\pref{itMonCTwo}, we see that $\ug n = \Chgbases bc n$, where $c > \ug(n-1)$, i.e.~we added a base between $\ug(n-1)$ and $\ug n$, which by monotonicity, yields $\ug m <c $.
But that means that $ \ug m$ is now available as a $c$-digit, so we can indeed substitute it for the rightmost $c$ in $\ug n $, thus obtaining $c^{\ug m} < \ug n $ with $u_C(c^{\ug m}) = u_C(\ug n)[i+1] = u_B( n)[i+1]$.
In the countable cofinality case, we simply replace the rightmost coefficient $a$ by $a'$ with $u_C(a') = u_C(a)[i+1]$, assuming inductively that $C$ can also compute the fundamental sequences for $a$.
This idea of replacing the rightmost ocurrence of a base or coefficient to simulate the fundamental sequences is made precise in the next lemma.

\newcommand{\diffn}{\Delta}

\begin{lemma}\label{lemmUCFS}
Let $B$ be a base hierarchy, $ b\geq 2$, and $s\in\mathbb N$ be in $U^b_B$-normal form.
Set $U = U^b_B$.

\begin{enumerate}

\item\label{itUCSFSucc} If $U(s) = \xi+1$, then $s>0$, $U(s-1) = \xi$, and $s-1$ is in $U^b_B$-normal form.

\item\label{itUCSFLimit} If $U(s)$ is a limit, suppose that $\iota <\tau(U (s)) $ and $\diffn \in (0, \min B]$ are such that $  \iota = u_B(c)$ for some $c\leq b-\diffn$ in $u_B$-normal form, and either $\iota$ is infinite, or for every $b$-digit $a$ of $s$, if $U(a)$ is a limit, then there is $a'\leq a -\diffn$ in $u_B$-normal form such that $u_B(a') = u_B(a)[\iota]$.

Then, there is $s'< s$ in $U^b_B$-normal form such that $U ( s' ) =   U (s)  [ \iota ] $ and $s' \leq s-\diffn $.

\end{enumerate}
\end{lemma}

\begin{proof} 
We prove both claims simultaneously by induction on $s+b$. 
\smallskip

\begin{Cases}
\item ($s < b$). Then, $U(s) = u_B(s) < \Om$ and thus has countable cofinality, so $\iota $ is finite.
Consider the following sub-cases.
\medskip

\begin{Cases}
\item ($U(s)$ is a successor). In this case, either $s<\min B$, $U(s) = s $, and $U(s)[i] = s-1$ by Lemma~\ref{lemmFS}\pref{itFSSucc}, so we can take $s'=s-1 $, or else $b':= \Base_B(s) <b $ and we can again take $s'=s-1$ by the induction hypothesis given that $s+b'<s+b$.
\medskip

\item  ($U(s)$ is a limit). By assumption, we can find $s' \leq s -  \diffn$ in $u_B$-normal form (which, since $s'<b$, implies that $s'$ is in $U^b_B$-normal form by definition) with $u_B(s') = u_B(s) [\iota]$, and since $s'<s <b$, we have that $U(s')=u_B(s')$, so
\[U(s') = u_B(s') = u_B(s) [\iota] = U(s)[\iota].\]
\end{Cases}\medskip

\item ($s \geq b$). Write  $s=_b b^ea+r$ and $U(s)=_\Om \Om^\eta\al+\ga$.\medskip

\begin{Cases}
\item ($\ga>0$). Then, $ U(   r) = \ga$.
We apply the induction hypothesis to find suitable $r'$ with either $r'=r-1$ if $U(s)$ (hence $U(r)$) is a successor, and otherwise $r'\leq  r- \diffn$ if $U(r)$ is a limit, and set  $s' = { b}^{  e} a +r' $, which respectively is either $s-1$ or at most $s-\diffn$.
Since $U(r')<U(r)$ and  $\Om^{U(e)} u_B(a) +U(r)$ is in $\Om$-normal form, we easily deduce that $\Om^{U(e)} u_B(a) +U(r')$ is in $\Om$-normal form, from which we readily see that $ s'$ is in $U^b_B$-normal form.
\medskip

\item ($\ga=0$). In this case, $U(s)$ is a limit, so we only need to establish the second claim.
Moreover, $  U(a)=u_B(a) =\al $.
We consider cases according to $\al$.
\smallskip

\begin{Cases}
\item  ($ \al $ is a limit).
The assumption on digits yields $a' \leq  a -\diffn$ in $u_B$-normal form such that $\al[\iota] = u_B(a)[\iota] =  a' $, and thus we may set
\[s' =  b^{  e}a' \leq b^{  e}a - b^e \leq b^{  e}a - \min B \leq s -\diffn \]
(where the second inequality uses $e\geq 1$ since $s\geq b$) and see that
\[U(s') = \Om^{U(e)}u_B(a') = \Om^{U(e)}(u_B(a )[\iota]) = ( \Om^{U(e)} u_B(a )) [\iota] = U(s)[\iota]. \]
The expression $\Om^{U(e)}u_B(a') $ is trivially in $\Om$-decomposition, and we conclude that $s'$ is in $U^b_B$-normal form.
\medskip

\item ($1<\al = \de+1 $). Then, $(\Om^\eta\al)[\iota] = \Om^\eta\de + \Om^\eta[\iota]$.
Let $a' = a-1$ and $b'=\Base_B(a)\leq b$.
Since $a'+b'<s+b$, the induction hypothesis applied to the first claim yields $U_B(a)=U^{b'}_B(a)=U^{b'}_B(a') +1 $.
Since $U_B(a) $ is not a limit, by Lemma~\ref{lemmULimit}\pref{itULimitMin}, $\min B\nmid a$, and thus $a\notin B$.
It follows that $a'\geq b'$, so $b'=\Base_B(a')$ and thus $U_B(a')=U^{b'}_B(a')$, hence $u_B(a) = \de$.
We must also have that $  U(b^e)  = \Om^\eta $, so the induction hypothesis yields suitable $r'\leq b^e-\diffn$ so that $b^ea'+r'\leq s-\diffn$ and
\[(\Om^\eta(\de+1))[\iota] = \Om^\eta\de + \Om^\eta[\iota] = U (b^ea'+r').\]
The $\Om$-decomposition condition is deduced from $U(r') < U(r)$ as above.
\medskip

\item($\al = 1$). We see  that $ U (e) = \eta$, so we apply the induction hypothesis to find suitable $e'<e$ with $u_B(e') = \eta [\iota]$.
Let $a'=1$ if $\eta$ is a limit, otherwise, $a' =c$.
We may then set $s'=b^{e'}a'$, and once again, $\Om^{U(e')}u_B(a')$ is in $\Om$-decomposition trivially.
If $\eta$ is a limit, then $e' \leq e-\diffn$, so
\[b^{e'} \leq   b^{e-\diffn} \leq b^e -\diffn ,\]
where the second inequality is immediate from the fact that $x\mapsto b^x$ is strictly increasing.
Otherwise, 
\[b^{e'}c \leq b^{e'}(b-\diffn) \leq b^e -\diffn .\qedhere\]
\end{Cases}
\end{Cases}
\end{Cases}
\end{proof}

The general picture is that as we iteratively apply the upgrade from $\odh_i$ to $\odh_{i+1}$, we can compute more and more of the fundamental sequences.
As a base case, we note that we can always compute them up to $0$.

\begin{lemma}\label{lemmFSZero}
Let $B$ be any base hierarchy and $n>0$ be in $u_B$-normal form.
Then, there is $n'<n$ in $u_B$-normal form such that $u_B(n') = u_B(n)[0]$.
If, moreover, $u_B(n)$ is a limit, then we can take $n' \leq n-\min B$.
\end{lemma}

\begin{proof}
Proceed by induction on $n$.
If $n<\min B$, then take $n'=n-1$, and note that both are in $u_B$-normal form by definition.
If $\tau(U_B(n)) = \Om$, then $ u_B(n)[0] = 0$ and we can take $n'=0$, which is in $u_B$-normal form.
Since $n\geq \min B$, $0\leq n-\min B$.
Otherwise, we may apply Lemma~\ref{lemmUCFS} to $U_B(n)$, where in the limit case we set $\iota=0$ and $\Delta=\min B $.
\end{proof}

The situation for $i>0$ is more subtle.
Recall that we write $\ug_i$ for $\ug_{\odh_i}^{\odh_{i+1}}$.
Then, $\odh_1$ can compute fundamental sequences up to $1$, but only for numbers of the form $\ug _0 n $.
Similarly, $\odh_1$ can compute fundamental sequences up to $2$, but only for numbers of the form $\ug _1\ug _0 n $.
More generally, $\odh_i$ can compute fundamental sequences up to $i$, but only to numbers that have been successively upgraded $i$ times.

In order to prove this uniformly, it is convenient to identify a key structural property of $\ug \mathbb N$, i.e., of the image of the upgrade operator.
Let $B$ be a base hierarchy and $X\subseteq \mathbb N$.
We say that $X$ is {\em $B$-digit closed} if whenever $n\in X$, $b:=\Base_B(n)\in X$, and if $d\in\coeffs b n $, it follows that $d\in X$.
Obviously, $\mathbb N$ is always $B$-digit closed, and the following is readily verified using Lemma~\ref{lemmUgProps}\pref{itUgPropsCoeffs}.

\begin{lemma}\label{lemmDigitClosed}
Let transformations be from a base hierarchy $B$ to a good successor $C$.
Then, if $X\subseteq \mathbb N$ is $B$-digit closed, it follows that $\ug X$ is $C$-digit closed.
\end{lemma}

The takeaway is that if $B$ computes fundamental sequences up to $i$ for any $n\in X$, where $X$ is $B$-digit closed, then $B^+$ computes fundamental sequences up to $i+1$ for any $n\in \ug X$.
The following lemma makes this precise.
Below, note that neither $m'\in X$ nor $n'\in \ug X$ are required (and, in general, this will not be the case).
 
\begin{lemma}\label{lemmOuroFS}
Let transformations be from a base hierarchy $B$ to $C=B^+$.
Suppose that $X\subseteq \mathbb N$ is $B$-digit closed, and $i >0$ and $\Delta\leq\min B$ are such that for every $m\in X$, $m$ is in $u_B$-normal form and there exists $m'<m$ in $u_B$-normal form such that $u_B(m')= u_B(m)[i-1]$, and such that if $u_B(m) $ is a limit, then $m'\leq m-\Delta$.

Then, if $n \in X $, there is $n'<\ug n$ such that $n'$ is in $u_C$-normal form and $u_C(n') = u_B(n )[i]$.
Moreover:
\begin{enumerate}

\item \label{itOuroFSSucc} If  $u_B(n ) $ is a successor, then $n'=n-1$.

\item \label{itOuroFSLim} If $u_B(n ) $ is a limit, then we can take $n'\leq n-\Delta$.

\end{enumerate}
\end{lemma}

\begin{proof}
Proceed by induction on $n$.
Let $b=\Base_B (n)$, $\xi= u_C(\ug n)$, and $c=\Base_C(\ug n)$.
If $ n <\min B$, then we must have $\xi[i] = n-1$, so we can take $n'=n-1$.
If $n\in B$, then $\xi=\om $ and we take $n' = i =\om[i]$.
Note that the assumption that $u_B(\min B)[i-1] = u_B(m')$ for some $m'\leq \min B-\Delta$ yields $i-1\leq\min B-\Delta$, so $i \leq \min B+1-\Delta=\min C-\Delta\leq \ug n-\Delta$.
 
Otherwise, let $\zeta = U_B(n)$ and note that $\zeta>\Om$, since $n\notin B$.
By Proposition~\ref{propPresU}, $\zeta  = U_C(\ug n)$, and $ \xi = \psi(\zeta )$ is in $\psi$-normal form by Lemma~\ref{lemmMC}.
Then, $\xi[i] = \psi(\zeta' ) $ with $\zeta' = \zeta[\iota]$, where either $\tau(\zeta) = \om$ and $\iota = i$, or else, $\tau(\zeta) = \Om$ and $\iota = \xi[i-1]$ (unless $\tau(\zeta) = \Om$ and $ i= 0$, but in this case we can just set $n'=0$).
Then, consider the following cases.
\smallskip

\begin{Cases}
\item ($\zeta'<\Om$).
Note in this case that $\zeta$ is a limit, for otherwise we would have $\zeta= \zeta'+1<\Om$, contrary to assumption.
Moreover, $\zeta'   \in \{0,1\}$ by Lemma~\ref{lemmZetaIota}\pref{itZetaIiotaUnc}, and thus $\xi[i] =  \psi(\zeta') \in\{1,2\} $.
Note that since $B$ is a base hierarchy, $\min B\geq 2$, so $\min C\geq \min B+1 \geq 3$, and hence we can take $n'= \xi[i] <\min C$, which is in $u_C$-normal form by default.
Since $n\notin B$, $n\geq \min C +1\geq \min B+2$, and hence $n' \leq 2 \leq n-\min B \leq n-\Delta $.
\medskip

\item ($\zeta'\geq  \Om$). By Lemma~\ref{lemmDigitClosed}, $\ug X$ is $C$-closed, given that $X$ is $B$-closed. Thus by the induction hypothesis, $U^{c}_{C}( \ug n)$   satisfies the assumptions of Lemma~\ref{lemmUCFS}\pref{itUCSFLimit} regarding $c$-digits. With this in mind, consider two sub-cases according to the cofinality of $\zeta$.\smallskip
\begin{Cases}

\item ($\tau ( \zeta)=\om)$.
In this case, $\zeta' = \zeta[i]$. 
By Lemma~\ref{lemmUCFS}, there is $n'$ in $U^c_C$-normal form with either $n'=n-1$ if $\zeta$ is a successor or $n' \leq \ug n -  \Delta$ if $\zeta $ is a limit, and  such that $U^c_C(n') = \zeta[i]$.
If $n'<c$, then $n'$ is in $u_C$-normal form since it is a $c$-digit in $U^c_C$-normal form.
Otherwise, $c\leq n' <n$, so $\Base_C(n') = c$.
Thus, $U _C(n') = U^c_C(n') = \zeta[i]$, and $n'$ is in $U _C$-normal form.
Moreover, $u_C(n') = \psi( \zeta[i])$, which is in $\psi$-normal form by Lemma~\ref{lemmFS}\pref{itFSCount}, so Lemma~\ref{lemmMC} implies that $n'$ is in $u_C$-normal form, as needed.
\medskip

\item ($\tau ( \zeta) =\Om)$.
In this case, $\xi$ must be a limit and $\zeta' =\zeta[\xi[i-1]]$.
Let $c'=\Base_C(\ug {n-1})$.
We claim that there is $ s < c $ such that $ s  $ is in $u_C$-normal form and $u_C(s) = \xi[i-1] $.
The assumption on $B$ yields $s'\leq n-\Delta$ such that $u_B(s') = u_B(n)[i-1]$ and $s'$ is in $u_B$-normal form.
Taking $s = \ug s'$, we see by Proposition~\ref{propPresU} that
\[u_C(s ) = u_B(s ') = u_B(n)[i-1]= u_C(\ug n)[i-1]=\xi[i-1]\]
and $s$ is in $u_C$-normal form.
Moreover, since $s'\leq n-\Delta$, we have that $s'\leq (n-1)-(\Delta-1)$, so 
\[s=\ug s' \leq \ug( (n-1)-(\Delta-1)) \leq \ug (n-1)-(\Delta-1) \leq  c  - \Delta ,\]
where the first two inequalities use Lemma~\ref{lemmMonUg}\pref{itMonUgOne} and the third uses Proposition~\ref{propMonC}\pref{itMonCTwo},
so $U^{c }_{C}(\ug n )$ satisfies the assumptions of Lemma~\ref{lemmUCFS} with $\iota := u_C(s)$.
Hence, there is $n' \leq \ug n -\Delta $ such that $U^{c}_{ C} (n') = \zeta[\xi[i-1]] = \zeta' $ and $n'$ is in $U^{c}_C$-normal form.
If $n'<c$, then it must be in $u_C$-normal form.
Otherwise, $\Base_C(\ug n)=c \leq n'<n$, so $\Base_C(n')=c$ and $n'$ is in $U_C$-normal form.
Finally, $u_C(n') = \psi(\zeta')$, which is in $\psi$-normal form by Lemma~\ref{lemmFS}.
We conclude that $n'$ is in $u_C$-normal form, as required.\qedhere
\end{Cases}
\end{Cases}
\end{proof}
 
Henceforth, we will focus our attention on the greedy dynamical hierarchy $\odh=(\odh_i)_{i\in\mathbb N}$.
Recall that when a dynamical hierarchy is fixed, we may write $\ug_i $ instead of $\ug_{\mathcal B_i}^{\mathcal B_{i+1}}$.
More generally, for $\ell\geq i$, we will define $\ug_i^\ell $ by letting $\ug_i^i $ be the identity and $\ug_i^{\ell+1} = \ug _\ell \circ \ug_i^{\ell} $.
As the next proposition shows, as $\ell$ increases, $\odh_{\ell}$ can compute more and more of the fundamental sequences of $\ug_i^{\ell}n $.

\begin{proposition}\label{propOuroFS}
Let $\mathcal D$ be the greedy dynamical base hierarchy.
Write $u_i$ for $u_{\mathcal D_i}$ and let $i,j\in\mathbb N$.
Suppose that $n>0 $ is in  $u_i$-normal form.
Then, there is $n'<\ug_i^{i+j} n$ such that $n'$ is in $u_i$-normal form and $u_{i+j}(n') = u_i(n )[j]$.
If moreover $u_i(n ) $ is a limit, then we can take $n'\leq n-i-2$.
\end{proposition}
 
\begin{proof}
For $j\in\mathbb N$, let $X_j = \ug^{i+j}_i\mathbb N$; we must find suitable $n'$ for $n\in X_j$.
For $j=0$, we have that $X_0 =\mathbb N$, so $n'$ can be found by Lemma~\ref{lemmFSZero}, given that $i+2=\min\odh_i$.
Otherwise, assume inductively the claim for all $m\in X_j$.
Note that $X_{j+1} = \ug_j X_j$, so
Lemma~\ref{lemmOuroFS}\pref{itOuroFSLim} then yields the desired $n'$.
\end{proof}

We are now ready to prove that the termination of the Goodstein process for the greedy dynamical hierarchy is independent of $\sf KP$.
Below, $\se xy$ refers to the superexponential function, given by $\se x 0=1$ and $\se x {y+1} = x^{\se x y}$, and $[x]_{y}$ denotes the remainder of $x$ modulo $y$.

\begin{proof}[Proof of Theorem~\ref{theoGoodInd}]
Consider the function $G \colon\mathbb N\to \mathbb N$ such that $G(k)$ is the termination time of the Goodstein sequence for $\odh$ starting on $\se 2 {k}+1$.
Clearly, Theorem~\ref{theoTerm} proves that $G$ is a total function over $\sf KP$ (or much weaker systems).
We show that $k\mapsto G(\se 2 {k}+1)$ grows faster than the function $F $ of Definition~\ref{defFFun}, and thus $G$ is not provably total in $\sf KP$ by Theorem~\ref{theoKPInc}; hence, Theorem~\ref{theoTerm} is not provable in $\sf KP$.

Fix $k$ and let $I= F(k)$.
For $i<I$, let $m_i = \mathbb G^\odh_i(\se 2 {k} +1 )$ and $\xi  = u_1(m_1) = \se \Om {k}$.
Letting $u_i = u_{\odh_i}$, we show that there are sequences $(t_i)_{i<I}$ and $(n_i)_{i<I}$ of natural numbers, where $(t_i)_{i<I}$ is increasing and for every $i<I$, $u_{t_i}(n_i ) = \xi \llbracket i\rrbracket $.
This will yield $  m_{t_i}>0$ whenever $i<I$, and thus $G( \se 2 {k}+1)\geq F(k) $.

To be precise, we require the following for all $i<I$:
\begin{multicols}2
\begin{enumerate}

\item\label{itGoodIndOne} if $j<i$, then $t_j<t_{i}$; 

\item\label{itGoodIndOneb} $ n_i $ is in $u_{t_i}$-normal form;

\item\label{itGoodIndTwo} $u_{t_i}(n_i ) = \xi \llbracket i\rrbracket $;

\item\label{itGoodIndTwob} $n_i \in \ug_{t_i-i}^ {t_i} \mathbb N$;

\item\label{itGoodIndThree} letting $b_i:=t_i+2$, we have that

$t_i\geq 2i+[n_i]_{b_i} + 1$, and

\item\label{itGoodIndFour}  $ n_i \leq m_{t_i}$.

\end{enumerate}
\end{multicols}
\noindent Conditions \pref{itGoodIndOne}, \pref{itGoodIndTwo}, and \pref{itGoodIndFour} are sufficient to obtain the desired lower bound, but the others are needed to ensure progress of the recursive construction.

We define $n_i$ recursively as follows.
Set $n_0=m_1$ and $t_0=1$, which are readily checked to satisfy all required properties.
Then, suppose that $i+1<I$ and $n_i$ has been defined and satisfies \pref{itGoodIndOne}-\pref{itGoodIndFour}.
By \pref{itGoodIndTwob},    $n_i  \in \ug _{t_i-i } ^ {t_i} \mathbb N $,  and by \pref{itGoodIndThree}, $ t_i-i  \geq i+1$.
Using Proposition~\ref{propOuroFS}, choose $s_0<\ug_{t_{i} }  n_i$ in $u_{t_i+1}$-normal form such that $u_{t_i+1 }(s_0) =  u_{t_i }(n_i )[i+1] $, with $s_0= \ug_{t_{i} } n_i-1$ if $u_{t_i }(n_i )$ is a successor and $s_0 \leq \ug_{t_i} n - i-3$ if $u_{t_i }(n_i )$ is a limit.
Since, by \pref{itGoodIndTwo}, $u_{t_i }(n_i ) = \xi\llbracket i\rrbracket $, $u_{t_i+1 }(s_0) =  \xi\llbracket i + 1\rrbracket $.

Now, consider two cases.
\smallskip

\begin{Cases}

\item ($u_{t_i }(n_i ) $ is a successor).
Set $n_{i+1} = s_{0}$ and $t_{i+1} = t_{i }+1$.
We check that all required properties hold.
\smallskip

\noindent \pref{itGoodIndOne} Clearly, $t_{i+1} > t_{i } $.
\smallskip

\noindent \pref{itGoodIndOneb}-\pref{itGoodIndTwo} Follow from the fact that $s_{0}$ was chosen to be $u_{t_{i }+1}$-normal form with $u_{t_i+1 }(s_0) =  \xi\llbracket i + 1\rrbracket $.
\smallskip

\noindent \pref{itGoodIndTwob} 
Use~\pref{itGoodIndTwob} inductively to write $n_i = \ug_{t_i-i}^ {t_i} (x+1)$. Lemma~\ref{lemmMonUgRem}\pref{itMonUgRemSucc} and backwards induction on $j \leq i$ show that $\ug_{t_i-i}^ {t_i-j+1} (x+1) = \ug_{t_i-i}^ {t_i-j+1}  x+1 $, so for $j=0$, we obtain  $n_{i+1}+1 = \ug_{t_i-i}^ {t_i+1}  x+1 $, thus $n_{i+1}  = \ug_{t_i-i}^ {t_i+1}  x  $.\smallskip

\noindent \pref{itGoodIndThree} We have that $ \ug_{t_i} n_i$ is not a multiple of $b_{i+1}$ by Lemma~\ref{lemmULimit}\pref{itULimitMin}, so
 $[n_{i+1}]_{b_{i+1}} = [ \ug_{t_i} n_{i}]_{b_{i+1}}-1 = [  n_{i}]_{b_{i}}-1 $, where the last equality uses Lemma~\ref{lemmMonUgRem}\pref{itMonUgOne}.
We thus have
\begin{align*}
t_{i+1} &= t_i+1 \geq 2i+ [  n_{i}]_{b_{i}} + 2  = 2(i+1) +[n_{i+1}]_{b_{i+1}} +1 .
\end{align*}

\noindent \pref{itGoodIndFour} 
From $n_i\leq m_{t_i}$ and monotonicity (Lemma~\ref{lemmMonUg}\pref{itMonUgOne}), we obtain
\[n_{i+1} = \ug_{t_i} n_i-1 \leq \ug_{t_i} m_{t_i} - 1 = m_{t_i+1} = m_{t_{i+1}}. \]

\medskip

\item ($u_{t_i }(n_i ) $ is a  limit).
Then, $s_0 \leq  n_i - i-3$.
For $j>1$, define $s_j  = \ug_{t_i+1}^{t_i+j} s_0$.
Set $t_{i+1} = t_i+ i+3$ and $n_{i+1}=s_{i+2}$.\smallskip

\noindent \pref{itGoodIndOne} Clearly, $t_{i+1} > t_i $.
\smallskip

\noindent \pref{itGoodIndOneb} This holds in view of Proposition~\ref{propPresU}, given that $s_{0}$ is in $u_{t_{i }+1}$-normal form, so induction shows that each $s_j$ (including $s_{i+2}=n_{i+1}$) is in $u_{t_{i }+j}$-normal form.
\smallskip

\noindent \pref{itGoodIndTwo} Proposition \ref{propPresU} and induction on  $j$  yield
\[\xi \llbracket i+1\rrbracket = u_{t_i}(n_i) [i+1] =  u_{t_i +1}(s_0) = u_{t_i +j}(s_j) = u_{t_{i+1}}(n_{i+1}) .\]

\noindent \pref{itGoodIndTwob}
This holds since $n_{i+1} = \ug_{t_i+2}^{t_i+i+3} s_1$.\footnote{Note that we purposefully use $s_1$ and not $s_0$ in this step, since the latter would force us to set $t_{i+1} = t_i+i+2$, which is too small for the subsequent \pref{itGoodIndThree}.}
 \smallskip

\noindent \pref{itGoodIndThree}
By Lemma~\ref{lemmZetaIota}\pref{itZetaIiotaC}, $u_{t_i}(n_i) [i+1] = \om\al+r $ for some $r\leq i+1$.
Lemma~\ref{lemmOuroFS}\pref{itOuroFSSucc} applied $r$ times yields $n_{i+1} = p+r $ with $u_{t_{i+1} }(p) = \om\al$, and thus by Lemma~\ref{lemmULimit}\pref{itULimitMin}, $p$ is a multiple of $\min \odh_{t_{i+1}} =b_{i+1}$.
We conclude that $[n_{i+1}]_{b_{i+1}} = r \leq i+1$.
Then, since $t_i\geq 2i+1$,
\[t_{i+1} = t_i+ i+3 \geq 3i +4 =3(i+1)+1 \geq 2(i+1) + [n_{i+1}]_{b_{i+1}} +1.\]
 \smallskip

\noindent \pref{itGoodIndFour} We show inductively using Lemma~\ref{lemmMonUg}\pref{itMonUgOne} that for $0\leq j\leq i+2$, $s_{j }  \leq m_{t_i+j} - i-2  +j $.
For the base case, we use \pref{itGoodIndFour} to see that $n_i\leq m_{t_i}$, and thus using Lemma~\ref{lemmMonUg}\pref{itMonUgOne},
\begin{align*}
s_{0}    \leq  \ug_{t_i }n_i - i-3  
  \leq \ug_{t_i} m_{t_i} - i-3  =  m_{t_i+1} - i-2.
\end{align*} 
For the inductive step,
\begin{align*}
s_{j+1} & = \ug_{t_i+j} s_{j} \leq  \ug_{t_i+j}(m_{t_i+j} - i-2  +j) \\
& \leq \ug_{t_i+j} m_{t_i+j} - i-2  +j =  m_{t_i+j+1} - i-2  +j+1.
\end{align*} 
Setting $j=i+2$, we see that $n_{i+1} = s_{i+2} \leq m_{t_i+i+2} = m_{t_{i+1}}$.\qedhere
\end{Cases}
\end{proof}

\section{Finite Base Hierarchies}

The proof of Theorem~\ref{theoGoodInd} uses a dynamical hierarchy whose base hierarchies (aside from the first) are infinite.
One may wonder if the use of infinite base hierarchies is essential.
The answer is {\em no;} here we will briefly sketch how our constructions should be modified to show this.
Below, say that a {\em finite dynamical hierarchy} is a dynamical hierarchy $\mathcal B$ such that $\mathcal B_i$ is finite for all $i$.

\begin{customthm}{\ref{theoGoodInd}'}\label{theoPrime}
$\sf KP$ does not prove that if $\mathcal B$ is any finite dynamical hierarchy and $m\in\mathbb N$, then there is $i\in \mathbb N$ such that $\G^{\mathcal B}_i(m) = 0$.
\end{customthm}

\begin{proof}[Proof sketch]
Fix $m\in \mathbb N$; we will modify the construction of Definition~\ref{defCanon} to obtain a finite dynamical hierarchy $\mathcal D' = \mathcal D'(m)$ such that $\mathbb G^{\mathcal D'}_i(m ) =\mathbb G^{\odh_i}(m ) $ for all $i$.
To do this, let $\mathcal D'_0 = 2$ and $k_0 = m$.
Then, using the notation of Definition~\ref{defCanon}, define inductively $\mathcal D'_{i+1} = (\mathcal D'_i)^+_{k_i} $ and $k_{i+1} =\ug_{\mathcal D'_i}^{\mathcal D'_{i+1}}k_i$; Lemma~\ref{lemmRestrict} and induction show that indeed $\mathbb G^{\mathcal D'}_i(m ) =\mathbb G^\odh_i(m ) $ for all $i$.
Setting $m=\se 2{n}+1$, it follows that $\mathbb G^{\mathcal D'}_i(m) = 0 $ implies that $i>F(n)$, so the statement is not provable in $\sf KP$.
\end{proof}

Note, however, that $\mathcal D'$ depends on $m$, so one may still ask if there is a {\em single} finite dynamical hierarchy which yields independence.
Indeed, such a dynamical hierarchy can be constructed by diagonalization.

\begin{customthm}{\ref{theoGoodInd}$^\star$}\label{theoPlus}
There is a finite, computable dynamical hierarchy $\mathcal D^\star$ such that $\sf KP$ does not prove that if $m\in\mathbb N$, then there is $i\in \mathbb N$ such that $\G^{\mathcal D}_i(m) = 0$.
\end{customthm}

\begin{proof}[Proof sketch]
Similar to the above construction, but now set $\mathcal D^\star_0 = \{2\}$ and $k_0 = 2$, then $\mathcal D^\star_{i+1} = (\mathcal D^\star_i)^+_{k_i} $ and $k_{i+1} =\se{\max \mathcal D^\star_{i+1}}{i+1}$.
Then, $u_{\mathcal D^\star_i}(k_i) = \psi(\se \Om i) $.
Now, let $m = k_{i+1}+i$.
Since $\ug_ j  x -1\geq x-1 $ holds for all $x$ and $j$, one readily checks that $\mathbb G^{\mathcal D^\star}_i(m ) \geq k_{i+1}$.
Then, we can reason as in the proof of Theorem~\ref{theoGoodInd} to see that if $\mathbb G^{\mathcal D^\star}_\ell (m) = 0$, then $\ell \geq F(i)$.
\end{proof}

\section{Concluding remarks}\label{secConc}

We have exhibited a Goodstein process based only on elementary functions, yet independent of very strong theories.
This opens a new pathway to obtaining powerful extensions of Goodstein's classic result.
A Goodstein principle of $\Gamma_0$-strength (much weaker than $\vartheta(\ve_{\Om+1})$) is obtained by using a version of the Ackermann function instead of the exponential, but with singleton base hierarchies~\cite{FSGoodstein}; it would be interesting to see if the combination of the two approaches leads to substantially more powerful principles, thus avoiding the use of Hardy functions in the search for ever more powerful Goodstein processes.

A second direction for future work lies in optimizing of our independence result.
The primary source of complexity of fractal Goodstein processes lies in the base hierarchies involved; we have already established that only finitely many bases must be added at each step, but a simpler process may be obtained by adding bases more sparingly.
Indeed, we added new bases for each $n$ that is a multiple of its own base, but as observed in Remark~\ref{remOmCof}, this does not necessarily lead to $U_B(n)$ having uncountable cofinality.
We could instead consider a dynamical base hierarchy which only adds new bases when $\tau(U_B(n))=\Om$, which would lead to independence with base hierarchies smaller than the greedy ones used.
We have only resisted this approach due to it being less number-theoretically natural, but our proofs readily apply to such a variation.
However, just how sparse we can make the base hierarchies while maintaining independence remains an interesting open question, possibly leading to phase transition results~\cite{MeskensWeiermann}.

Finally, there is the issue of optimizing the ordinal assignment(s).
Our goal was to keep ordinal calculations as simple as possible, which are already non-trivial with our approach.
As we have mentioned, a tight ordinal interpretation would lead to more cumbersome computations.
It would moreover not improve on the independence results for $\sf KP$, as these are not sensitive to the difference between $\vartheta$ and $\psi$.
However, more fine-grained independence results could be obtained by a precise ordinal interpretation, such as independence for weaker theories with respect to restricted versions of our fractal Goodstein process, so a more refined ordinal analysis may be a viable direction for future work.
For now, we find that the novelty of a two-interpretation approach has greater added value, as we predict that it will become indispensable as Goodstein principles are pushed farther beyond the limits of predicativity.

\subsection*{Acknowledgements}

We thank Milan Morreel for many insightful comments and suggestions which have helped improve the article.

\bibliographystyle{amsplain}
\bibliography{biblio}

\end{document}